\documentclass[a4paper]{article}
\usepackage{amsmath,amssymb,amsthm,amsfonts,graphicx,fullpage}
\usepackage[english]{babel}
\usepackage{color}
\usepackage{url}
\usepackage{wrapfig}
\usepackage{graphicx}   
\usepackage{caption}
\usepackage{subcaption}  
\usepackage[title]{appendix} 
\usepackage{cite}
\usepackage{enumitem}
\usepackage{hyperref}
\usepackage{cleveref}
\usepackage{authblk}
\allowdisplaybreaks
\hypersetup{
	colorlinks=true,
	linkcolor=black,
	filecolor=black,      
	urlcolor=black,
	citecolor=black
}
\providecommand{\keywords}[1]
{
	\small	
	\textbf{Keywords --} #1
}

\newcommand{\norm}[1]{\left\lVert#1\right\rVert}

\newtheorem{assumption}{Assumption}[section]
\newtheorem{theorem}{Theorem}[section]
\newtheorem{lemma}[theorem]{Lemma}
\newtheorem{remark}{Remark}[section]
\newtheorem{proposition}[theorem]{Proposition}
\newtheorem{definition}{Definition}[section]

\begin{document}
	\title{\vspace{-4ex}\bf Efficacy of the Sterile Insect Technique in the presence of inaccessible areas: A study using two-patch models}
	\author[1]{P.A. Bliman} 
	\author[2, 1]{N. Nguyen} 
	\author[2]{N. Vauchelet}
	\affil[1]{\footnotesize MAMBA, Inria Paris; LJLL, Sorbonne University, CNRS, 5 Place Jussieu, 75005 Paris, France} 
	\affil[2]{\footnotesize LAGA, CNRS UMR 7539, Institut Galilée, University Sorbonne Paris Nord, 99 avenue Jean-Baptiste Clément, 93430 Villetaneuse, France} 
	\maketitle
	
	\begin{abstract}
		The Sterile Insect Technique (SIT) is one of the sustainable strategies for the control of disease vectors, which consists of releasing sterilized males that will mate with the wild females, resulting in a reduction and, eventually a local elimination, of the wild population. The implementation of the SIT in the field can become problematic when there are inaccessible areas where the release of sterile insects cannot be carried out directly, and the migration of wild insects from these areas to the treated zone may influence the efficacy of this technique. However, we can also take advantage of the movement of sterile individuals to control the wild population in these unreachable places. In this paper, we derive a two-patch model for \textit{Aedes} mosquitoes where we consider the discrete diffusion between the treated area and the inaccessible zone. We investigate two different release strategies (constant and impulsive periodic releases), and by using the monotonicity of the model, we show that if the number of released sterile males exceeds some threshold, the technique succeeds in driving the whole population in both areas to extinction. This threshold depends on not only the biological parameters of the population but also the diffusion between the two patches. 
	\end{abstract}
	
	\keywords{sterile insect technique, metapopulation model, monotone dynamical systems}
	\section{Introduction}
	Mosquitoes of genus \textit{Aedes aegypti} and \textit{Aedes albopictus} play a crucial role in transmitting various arboviruses to humans including dengue, chikungunya, and Zika virus. Existing treatments are only symptomatic, and available vaccines (e.g. Dengvaxia for dengue) have a lot of constraints \cite{CDC_DEN}. Consequently, the primary prevention lies in controlling the mosquito population \cite{WHO17}. However, traditional insecticide-based methods have limitations, prompting the need for innovative and sustainable strategies \cite{ACH19}, \cite{BEC20}.  Biological controls involve releasing large numbers of mosquitoes that are either sterile or incapable of transmitting diseases, which recently gained much attention. The Sterile Insect Technique is among these sustainable alternative methods which consist of the release of sterilized male mosquitoes that will mate with wild females \cite{KNI59}, \cite{DYC21}. These wild females, unable to lay viable eggs, will gradually drive the wild population to decline. The efficacy of SIT relies on a comprehensive understanding of the vector behavior, as well as accurate modeling of its dispersal, to optimize the release strategies. 
	
	Spatial heterogeneity in mosquito populations and mosquito-borne diseases occurs due to differences in the quality and quantity of their habitats, as well as variations in host density, temperature, and rainfall \cite{CHA11}, \cite{STO13}, \cite{MAR00}. Especially, the number and accessibility of sites where mosquitoes lay their eggs play a significant role in determining the size of adult mosquito populations by increasing the carrying capacity of the environment \cite{ABR15}. Developing models that capture mosquito behavior in response to environmental heterogeneity is crucial for designing effective control strategies, especially in the face of rapid global land-use changes.  Models using monotone dynamical systems were introduced (see e.g. \cite{ANG12}, \cite{DUM12}, \cite{STR19}) and applied efficiently  (see e.g. \cite{BLI19}, \cite{ANG20}, \cite{BLI22}) to study the SIT.  Not many mosquito modeling studies have incorporated migration or dispersal effects due to insufficient information on individual movement in the field as well as the complex analysis of models. Most of them used the diffusion approach, which considers space as a continuous variable. They were first developed in one-dimensional space using scalar reaction-diffusion equations \cite{MAN86}, \cite{LEW93}, then extended to sex-structured compartmental systems to consider the different behaviors of aquatic phases, wild females, males, and sterile males (see e.g.  \cite{ALM22}, \cite{ALM23}, \cite{NGA2})  and in higher dimension (see e.g. \cite{DUF}, \cite{NGA3}).  However, it remains challenging to explicitly incorporate the factors that affect the movement of sterile males. For instance, when resources are concentrated in patches or distinct locations, a metapopulation approach in which we treat space as a discrete set of patches and describe how the population on each patch varies with time is more suitable for modeling mosquito dispersal \cite{AUG08}, \cite{LUT13}, \cite{MAN17}.
	
	The application of the SIT in the field encounters a difficulty of the limitation in space when there are some inaccessible areas where people can not release sterile insects directly. For example, mosquitoes of the genus \textit{Aedes polynesiensis} primarily exploit land crab burrows for oviposition in certain French Polynesian atolls \cite{BON58}, \cite{LAR92}, \cite{LAR02}, \cite{HEA22}. The larvae in the crab burrows emerge into adult mosquitoes that can fly out to search for food and human blood for fertility. Another example is the inaccessibility in heavily forested or narrow mountainous terrain, where sterile insects must be released via helicopters owing to their better maneuverability \cite{VAR95}. However, one standout advantage of the SIT is that it relies on the natural ability of the male mosquitoes to move, locate, and mate with females. This behavior will take place in those areas that cannot be reached with conventional control techniques (i.e. insecticides). Therefore, we are interested in the mosquito population dynamics in the presence of such reservoirs and the elimination of the whole population while considering that the released sterile males can fly into unreachable sites. The patchy models with discrete diffusion mentioned above are a useful approach to describe the mosquito dynamic taking into account the inaccessibility to the burrows. We develop a two-patch model and in each patch, we consider a monotone dynamical system inspired by the models in \cite{STR19} where the population is divided into different compartments characterizing the aquatic phase, wild females, wild males, and sterile males. Except for the aquatic phase, individuals in other states move between patches at specific rates. The SIT is only carried out in the first patch and only affects the second one through these natural movements. Two-patch models were used to study the same problem in \cite{YAN20}, where they considered a simple scalar equation to describe the population dynamics in each patch. Our model provides a better understanding of how the dynamics of each stage influence the result of the control method. However, the complexity of our system does not allow us to obtain the full analysis of the model like what has been done in \cite{YAN20}.
	
	In the present work, we are interested in how to guarantee the successful elimination of the SIT in both areas and how the diffusion rates as well as other biological parameters influence the efficacy. To tackle this problem, we focus on studying the global stability of the extinction equilibrium in our system. Results of global asymptotic behavior for the single-species model depending on the discrete diffusion were provided in the literature\cite{ALL87}, \cite{TAK89}, \cite{LU93}. Lyapunov's second method was used in \cite{LI10} to investigate the multi-species system with discrete diffusion. Many works have been done to design robust strategies for releasing sterile males to drive a population to elimination \cite{BLI19}, \cite{BLI22}. We extend these control strategies to our two-patch system and prove the sufficient conditions for both constant continuous and periodic impulsive releases to drive the whole system to extinction. We obtain that when the number of released sterile males exceeds some threshold, the populations in both the treated and the inaccessible zone reach elimination, and we show how this critical value depends on the diffusion rates between two areas and other biological intrinsic parameters. In the original mathematical model provided by Knipling \cite{KNI59}, the population elimination depends on an overflooding ratio of the number of sterile males released to the initial wild male population size. In our result, the threshold number of sterile males does not depend on the initial level of the wild population due to the global stability of the zero steady state. Our results may help estimate the possibility and the amount of sterile mosquitoes necessary to complete elimination in the presence of hidden, inaccessible reservoirs.
	
	The organization of the paper is as follows. In section \ref{sec:model}, we present the formulation of the two-patch model and prove the monotonicity of the systems and some other preliminary results that will be applied in our proofs. Section \ref{sec:natural} is devoted to the study of the system without sterile insects. In Theorem \ref{thm:equi1}, we provide conditions for the persistence and extinction of the wild population on each patch. In section \ref{sec:elimination}, we study the dynamics of mosquito population in the presence of the SIT with two release strategies: constant and impulsive releases. Theorem \ref{thm:main} presents sufficient conditions on the average number of sterile males released per time unit to drive the population to elimination. In section \ref{subsec:principle}, we present the principle of our method and then apply it to prove Theorem \ref{thm:main}. Section \ref{sec:dependence} is focused on the dependence of the critical number of sterile males on parameters. The results in \ref{subsec:largediffusion} show that when the diffusion rates are large, the dynamics of the whole system are the same as in the case when there is no separation between the two sub-populations. Then, Theorem \ref{thm:bio} shows that the critical number of released sterile males depends monotonically on the biological parameters. Finally, some numerical illustrations are provided in Section \ref{sec:numeric}. 
	\section{Model \label{sec:model}}
	In this section, we present the formulation of the model used to study the population dynamics in \ref{subsec:model2patch}. Then, in \ref{subsec:monotonicity}, we provide some preliminary results that will be used later in the present work.
	\subsection{Formulation of the model}
	\label{subsec:model2patch}
	Consider two patches and denote $E_i, F_i, M_i$, and $M_i^s$ respectively the density of aquatic phase (eggs, larvae, pupae), fertilized females, wild males, and sterile males on the patch $i$ depending on time $t$. We consider a two-patch model coupled by the diffusion terms as follows where the dynamic in each patch is inspired by the model in \cite{STR19, ANG20}
	\begin{subequations}
		\label{eqn:main2patch}
		\begin{align}
			\dot{E_1} &= bF_1\left( 1 - \frac{E_1}{K_1}\right) - (\nu_E + \mu_E) E_1, \label{eqn:E1} \\
			\dot{F_1} &= r\nu_E E_1 \frac{M_1}{M_1 + \gamma M_1^s} - \mu_F F_1 - d_{12} F_1 + d_{21} F_2, \label{eqn:F1} \\
			\dot{M_1} & = (1-r)\nu_E E_1 - \mu_M M_1 - \beta d_{12} M_1 + \beta d_{21} M_2, \label{eqn:M1} \\
			\dot{M_1^s} &= \Lambda - \mu_s M_1^s - \alpha d_{12} M_1^s + \alpha d_{21} M_2^s, \label{eqn:Ms1} \\
			\dot{E_2} &= bF_2\left( 1 - \frac{E_2}{K_2}\right) - (\nu_E + \mu_E) E_2, \label{eqn:E2} \\
			\dot{F_2} &= r\nu_E E_2 \frac{M_2}{M_2 + \gamma M_2^s} - \mu_F F_2 - d_{21} F_2 + d_{12} F_1, \label{eqn:F2}  \\
			\dot{M_2} & = (1-r)\nu_E E_2 - \mu_M M_2 - \beta d_{21} M_2 + \beta d_{12} M_1, \label{eqn:M2} \\
			\dot{M_2^s} &= - \mu_s M_2^s - \alpha d_{21} M_2^s + \alpha d_{12} M_1^s. \label{eqn:Ms2} 
		\end{align}
	\end{subequations}
	
	\noindent In \textit{Aedes} mosquitoes, fertilization takes place very quickly after the females have hatched (the males wait for the females to hatch near the breeding sites). So an important assumption we made in system \eqref{eqn:main2patch} is that all female mosquitoes can mate with either wild or sterile males. The variable $F_i$ only characterizes the density of females fertilized by wild males and $E_i$ denotes the density of viable offspring in the aquatic phase. On patch $i$, the total amount of offspring emerging per time unit is $\nu_E E_i$. These correspond to the birth of males and females, with the respective quantities $(1-r) \nu_E E_i$ and $r \nu_E E_i$. Among the $r\nu_E E_i$ females that hatch at each unit of time, a quantity $r\nu_E E_i \dfrac{M_i}{(M_i+\gamma M_i^s)}$ mate with a wild male and produce viable offspring (which corresponds to the first term of equations \eqref{eqn:E1} and \eqref{eqn:E2}); and the females that couple with a sterile male are no longer involved in the reproduction.
	
	The interpretation of the parameters used in the model, with $i, j \in \{1, 2\}$, is as below
	\begin{itemize}[leftmargin=.5cm]
		\item $\Lambda(t)$ is the number per time unit of sterile mosquitoes that are released at time $t$ on the first patch;
		\item $\gamma \in [0,1]$ characterizes the competitiveness of sterile males;
		\item $b>0$ is the birth rate; $\mu_E>0$, $\mu_M>0$, and $\mu_F>0$ denote the death rates for the mosquitoes in the aquatic phase, for adult males, and for adult females, respectively;
		\item $K_i$ is an environmental capacity for the aquatic phase on patch $i$, accounting also for the intraspecific competition;
		\item $\nu_E>0$ is the rate of emergence;
		\item $r\in (0,1)$ is the probability that a female emerges, then $(1-r)$ is the probability that a male emerges.
		\item  $d_{ij} > 0$ is the moving rate of female mosquitoes from patch $i$ to patch $j$; the fertile males and sterile males move slower but with proportional rates respectively $\beta d_{ij}$, $\alpha d_{ij}$  where typically $0 < \alpha < \beta < 1$ in practice.  
	\end{itemize}
	Mosquito life cycle characteristics in both zones are linked to the zone sizes, comparable to the distance a mosquito can travel. In this way, the overall homogeneity of the species can be preserved, despite the possible effects of evolution, due to permanent mixing, so it is relevant to consider the same biological parameters $b, \ r,\ \mu_E,\ \mu_F,\ \mu_M$ in the two patches. 
	
	We recall the basic offspring number of  the sub-population in one patch as introduced in \cite{STR19}
	\begin{equation}
		\mathcal{N} = \dfrac{br\nu_E}{\mu_F(\mu_E + \nu_E)}. 
		\label{eqn:bon}
	\end{equation}
	
	The persistence and extinction of the population in the patch depend strongly on the value of this number. In Section, \ref{sec:natural}, we will show that $\mathcal{N}$ is also the basic offspring number of the whole two-patch system.
	\subsection{Preliminary results}
	\label{subsec:monotonicity}
	First, we provide some definitions and denotations of the order used in the present work. 
	\begin{definition}
		A matrix $A \in \mathcal{M}^{m \times n}$ is called non-negative, denoted $A \geq 0$, if all of its entries are non-negative. 
		
		\noindent It is called positive, denoted $A > 0$, if it is non-negative and there is at least one positive entry. 
		
		\noindent It is called strictly positive, denoted $A \gg 0$, if all of its entries are strictly positive. 
	\end{definition}
	In the present work, we also use the above definition of order for vectors in $\mathbb{R}^n$. Next, we recall the definition of a Metzler matrix and present a property of a Metzler matrix that will be used in this paper.
    \begin{definition}
        A matrix $A = (a_{ij})$ is called Metzler of all the off-diagonal components are nonnegative, i.e. $a_{ij} \geq 0 $ for $i \neq j$. 

        A matrix $A \in \mathcal{M}^{n \times n}$ is irreducible if it is not similar via a permutation to a block upper triangular matrix.
    \end{definition}
	\begin{lemma}
		\label{lem:metzler}
		Assume that a square matrix $A$ is Metzler and irreducible, then $e^A$ is strictly positive.
	\end{lemma}
	\begin{proof}
		Since $A$ is Metzler, then there exists a constant $\delta  >0 $ large enough such that $A + \delta I$ is a non-negative matrix with a positive element on the main diagonal. Moreover, $A$ is irreducible so $A+ \delta I$ is also irreducible. Thus, $A + \delta I$ is primitive, that is, there exists an integer $n > 0$ such that $(A + \delta I)^n \gg 0$. Hence, we have $e^{A + \delta I} \gg 0$, and since $\delta I$ commutes with all matrices, one has $e^A = e^{A + \delta I} e^{-\delta I} \gg 0$.
	\end{proof}
	
	We present in this section the so-called Kamke \cite{COP65} or Chaplygin \cite{CHA54} lemma for a cooperative system (Lemma \ref{lem:chaplygin}). Then, we apply this lemma to show the monotonicity of system \eqref{eqn:main2patch} in Lemma \ref{lem:monotonicity}. 
	
	\begin{lemma}
		\label{lem:chaplygin}
		For any $n \in \mathbb{N}^*$, consider a smooth function $\mathbf{f}: \mathbb{R}^n \rightarrow \mathbb{R}^n$, and a vector function $\mathbf{u}(t)$ satisfying a differential equation 
		\[\dot{\mathbf{u}} = \mathbf{f}(\mathbf{u}).\]
		Moreover, we assume that the above system is cooperative, that is,
		\begin{equation}
			\frac{\partial f_i}{\partial u_j}(t) \geq 0, \qquad \text{ for } i \neq j, \ t > 0.
			\label{eqn:cooperative}
		\end{equation}
		If a vector function $\mathbf{v}(t)$ satisfies a differential inequality 
		$\dot{\mathbf{v}} \leq \mathbf{f}(\mathbf{v})$ then, for initial data $\mathbf{v}(0) \leq \mathbf{u}(0)$, we have $\mathbf{v}(t) \leq \mathbf{u}(t)$ for all $t > 0$. 
	\end{lemma}
	To apply this Lemma to system \eqref{eqn:main2patch}, we first define the following order in $\mathbb{R}^8$ as follows
	\begin{definition}
		\label{def:order}
		For any vectors  $\mathbf{u},  \mathbf{v} \in \mathbb{R}^8$, we define an order $\preceq$ such that $\mathbf{u} \preceq \mathbf{v}$ if and only if 
		\[
		\begin{cases}
			u_i \leq v_i & \text{ for } i \in \{1,2,3,5,6,7\}, \\
			u_i \geq v_i & \text{ for } i \in \{4,8\}. 
		\end{cases}
		\]
		Moreover, we write $\mathbf{u} \prec \mathbf{v}$ if $\mathbf{u} \preceq \mathbf{v}$ and $\mathbf{u} \neq \mathbf{v}$. 
	\end{definition}
	
	The following result shows the comparison principle of system \eqref{eqn:main2patch}
	\begin{lemma}
		\label{lem:monotonicity}
		By denoting $\mathbf{u} = (E_1, F_1, M_1, M_1^s, E_2, F_2, M_2, M_2^s) \in \mathbb{R}^8$, we can write system \eqref{eqn:main2patch} as the form $\dot{\mathbf{u}} = \mathbf{f}(\mathbf{u})$ with $\mathbf{f}$ is $\mathcal{C}^1$ in $\mathbb{R}^8$.  In the subset  $\{0 \leq E_1 \leq K_1\} \cap \{0 \leq E_2 \leq K_2\}$ of  $\mathbb{R}^8_+$, system \eqref{eqn:main2patch} is monotone in the sense that if a vector function $\mathbf{v}(t)$ satisfies a differential inequality $\dot{\mathbf{v}} \preceq \mathbf{f}(\mathbf{v})$ then, for initial data $\mathbf{v}(0) \preceq \mathbf{u}(0)$, we have $\mathbf{v}(t) \preceq \mathbf{u}(t)$ for all $t > 0$. 
	\end{lemma}
	\begin{proof}
		By changing the variable to $\widetilde{\mathbf{u}} = (E_1, F_1, M_1, -M_1^s, E_2, F_2, M_2, -M_2^s)$, we can write system \eqref{eqn:main2patch} as
		\[
		\widetilde{\mathbf{u}} = \widetilde{\mathbf{f}}(\widetilde{\mathbf{u}}).
		\]
		This system is cooperative since in $\{0 \leq E_1 \leq K_1\} \cap \{0 \leq E_2 \leq K_2\}$ of  $\mathbb{R}^8_+$, we have
		\[
		\frac{\partial \widetilde{f}_1}{\partial \widetilde{u}_2} = b\left( 1 - \frac{E_1}{K_1}\right) \geq 0, \quad \frac{\partial \widetilde{f}_1}{\partial \widetilde{u}_j} = 0 \ \text{ for any } j > 2,
		\]
		\[
		\frac{\partial \widetilde{f}_2}{\partial \widetilde{u}_1} = r \nu_E \frac{M_1}{M_1 + \gamma M_1^s} \geq 0, \quad  \frac{\partial \widetilde{f}_2}{\partial \widetilde{u}_3} = r \nu_E \frac{\gamma M_1^s}{(M_1 + \gamma M_1^s)^2} \geq 0, 
		\]
		\[
		\frac{\partial \widetilde{f}_2}{\partial \widetilde{u}_4} = r\nu_E E_1 \frac{\gamma M_1}{(M_1 + \gamma M_1^s)^2} \geq 0, \quad  \frac{\partial \widetilde{f}_2}{\partial \widetilde{u}_6} = d_{21} > 0, \quad \frac{\partial \widetilde{f}_2}{\partial \widetilde{u}_j} = 0 \ \text{ for } j \in \{5,7,8\}.
		\]
		\[
		\frac{\partial \widetilde{f}_3}{\partial \widetilde{u}_1} = (1-r)\nu_E \geq 0, \quad \frac{\partial \widetilde{f}_3}{\partial \widetilde{u}_7} = \beta d_{21} > 0, \quad  \frac{\partial \widetilde{f}_1}{\partial \widetilde{u}_j} = 0 \ \text{ for } j \in \{2, 4, 5, 6, 8\} ,
		\]
		\[
		\frac{\partial \widetilde{f}_4}{\partial \widetilde{u}_8} = \alpha d_{21} > 0, \frac{\partial \widetilde{f}_4}{\partial \widetilde{u}_j} = 0 \ \text{ for } j \in \{1, 2, 3, 5, 6, 7\}. 
		\]
		Similarly for $\widetilde{f}_i$ with $i > 4$, so $\widetilde{\mathbf{f}}$ is cooperative. 
		
		For any vector function $\mathbf{v}$ such that $\dot{\mathbf{v}} \preceq \mathbf{f}(\mathbf{v})$, by the same variable change, one has $\dot{\widetilde{\mathbf{v}}} \leq \widetilde{\mathbf{f}}(\widetilde{\mathbf{v}})$. The initial data $\mathbf{v}(0) \preceq \mathbf{u}(0)$ implies that $\widetilde{\mathbf{v}}(0) \leq \widetilde{\mathbf{u}}(0)$. Therefore, by applying Lemma \ref{lem:chaplygin}, one has $\widetilde{\mathbf{v}}(t) \leq \widetilde{\mathbf{u}}(t)$ for any $t > 0$ which is equivalent to $\mathbf{v}(t) \preceq \mathbf{u}(t)$. 
	\end{proof}
	\medskip 
	
	In order to define the solution of \eqref{eqn:main2patch}, we make some assumptions for the release function $\Lambda(t)$
	\begin{assumption}
		Assume that function $\Lambda(t)$ satisfies
		\begin{equation}
			\Lambda(t) = \Lambda_1(t) + \Lambda_2(t),
			\label{eqn:sumLambda}
		\end{equation}
		where $\Lambda_1 \in L^1_\mathrm{loc}(0,+\infty)$, $\Lambda_1(t) \geq 0$ for almost every $t$, and $\Lambda_2$ is a sum of Dirac masses with positive weights. More precisely, for $0 < t_1 < t_2 < \dots $, one has
		\begin{equation}
			\Lambda_2(t) = \displaystyle \sum_{i=0}^{+\infty} w_k \delta_{t_k}(t) \qquad \text{ with } w_k > 0.
			\label{eqn:comb}
		\end{equation}
		 Assume moreover that there exists a time $T > 0$ such that the average value of $\Lambda$ over any $T$-time interval is finite, that is, 
		\begin{equation}
			C_\Lambda := \frac{1}{T} \sup_{t \geq 0} \int_t^{t+T} \Lambda(s) ds < +\infty.
			\label{eqn:CLambda}
		\end{equation}
		\label{hyp:Lambda}
	\end{assumption}
	Assumption \ref{hyp:Lambda} is natural since in practice, the total amount of the sterile males released in a finite time interval is finite. The term $\Lambda_2$ corresponds to impulsive releases.
	
	The next result shows that any trajectory of system \eqref{eqn:main2patch} resulting from any non-negative initial data is bounded.
	\begin{lemma}
		\label{lem:bound}
		Let $\Lambda$ satisfy Assumption \ref{hyp:Lambda}. For any non-negative initial data $(E_1^0, F_1^0, M_1^0, M_1^{s,0}, E_2^0, F_2^0, M_2^0, M_2^{s,0})$, there exists a unique solution $(E_1, F_1, M_1, M_1^s, E_2, F_2, M_2, M_2^s)$ of system \eqref{eqn:main2patch} that is smooth on each interval $(t_k, t_{k+1})$ of the impulsive function $\Lambda_2$ defined in \eqref{eqn:comb} with $k = 0,1,\dots$, and it is non-negative. If $E_i^0 < K_i$ with $i = 1, 2$, then $E_i(t) \leq K_i$ for any $t > 0$. Moreover, for all $t \geq 0$, we have the uniform bounds
		\[
		F_1 + F_2 \leq \max \left\{  F_1^0 + F_2^0, \ C_F \right\}, \quad M_1 + M_2 \leq \max \left\{  M_1^0 + M_2^0, \ C_M \right\}, 
		\]
		where 
		\[
		C_F := \frac{r \nu_E (K_1 + K_2)}{\mu_F} , 
		\ C_M := \frac{(1-r) \nu_E (K_1 + K_2)}{\mu_M},
		\]
		and 
		\[
		M_1^s(t) + M_2^s(t) \leq \max\left\{ M_1^{s,0} + M_2^{s,0}, \frac{TC_\Lambda}{1 - e^{-\mu_s T}}\right\} + T C_\Lambda, 
		\]
		with $T$ and $C_\Lambda$ defined in Assumption \ref{hyp:Lambda}. 
		One also has 
		\[
		\limsup_{t \rightarrow +\infty} (F_1 + F_2)(t) \leq C_F, \quad  \limsup_{t \rightarrow +\infty} (M_1 + M_2)(t) \leq C_M, \quad,  
		\]
		and 
		\[
		\limsup_{t \rightarrow +\infty} (M_1^s + M_2^s)(t) \leq \frac{TC_\Lambda}{1 - e^{-\mu_s T}}+ TC_\Lambda =: C_{M^s}.
		\]
	\end{lemma}
	\begin{remark}
		In the case $\Lambda \in L^\infty(0,+\infty)$, one can let $T$ tend to zero and obtain that $C_\Lambda = \sup_{t > 0} \Lambda(t)$ and $\displaystyle \limsup_{t \rightarrow +\infty} (M_1^s + M_2^s) \leq \frac{C_\Lambda}{\mu_s}$. 
		The condition of $\Lambda$ that we made in Assumption \ref{hyp:Lambda} is weaker than the $L^\infty$ assumption since we also include impulsive releases, represented by a sum of Dirac masses.
	\end{remark}
	\begin{proof}[Proof of Lemma \ref{lem:bound}]	
		By denoting $\mathbf{u} = (E_1, F_1, M_1, M_1^s, E_2, F_2, M_2, M_2^s) \in \mathbb{R}^8$, we can write system \eqref{eqn:main2patch} as the form $\dot{\mathbf{u}} = \mathbf{f}(\mathbf{u})$ with $\mathbf{f} = (f_1, f_2, \dots, f_8)$ is $\mathcal{C}^1$ in $\mathbb{R}^8$. We easily observe that in the subset $\{0 \leq E_1 \leq K_1\} \cap \{0 \leq E_2 \leq K_2\}$ of  $\mathbb{R}^8_+$, one has $f_i(u_i = 0) \geq 0$ for $i = 1, \dots, 8$, and $f_1(E_1 = K_1) \leq 0$, $f_5(E_2 = K_2) \leq 0$. Therefore, the vector field $\mathbf{f}$ on the boundary of $\mathbb{R}^8_+ \cap \{0 \leq E_1 \leq K_1\} \cap \{0 \leq E_2 \leq K_2\}$ is inward or tangential, so this set is positively invariant.	
		
		From equations \eqref{eqn:F1} and \eqref{eqn:F2}, since for $i = 1, 2$,  we have $\dfrac{M_i}{M_i + \gamma M_i^s} \leq 1$, and 
		\[
		\dot F_1 + \dot F_2 \leq r \nu_E (E_1 + E_2) - \mu_F (F_1 + F_2). 
		\]
		Since $E_1, \ E_2$ are bounded then we deduce that 
		\[
		(F_1 + F_2)(t) \leq (F_1^0 + F_2^0) e^{-\mu_F t} + \frac{r \nu_E (K_1 + K_2)}{\mu_F}(1 - e^{-\mu_F t}) \leq \max \left\{  F_1^0 + F_2^0, \ \frac{r \nu_E (K_1 + K_2)}{\mu_F} \right\},
		\]
		for any $t \geq 0$. For $i = 1, 2$, one has $F_i \geq 0$, thus $F_i(t) \leq \max \left\{  F_1^0 + F_2^0, \ C_F \right\}$ for any $t \geq 0$.  Let $t$ go to infinity we get $\displaystyle \limsup_{t \rightarrow +\infty} (F_1 + F_2)(t) \leq C_F$. One obtains similarly the inequalities for $M_1, M_2$. 
		
		For $M_1^s$ and $M_2^s$, by denoting $X_s(t) = M_1^s(t) + M_2^s(t)$, then from equations \eqref{eqn:Ms1} and \eqref{eqn:Ms2}, one has
		\[
		\dot{X}_s(t) = -\mu_s X_s(t) + \Lambda(t). 
		\]
		For any integer $k$, by integrating both sides of this equality in $((k-1)T, kT)$ with $T$ defined in Assumption \ref{hyp:Lambda}, one gets 
		\[
		X_s(kT) = e^{-\mu_s T} X_s((k-1)T) + \int_{(k-1)T}^{kT} e^{-\mu_s (t -(k-1)T)} \Lambda(t) dt. 
		\]
		Since $e^{-\mu_s (t -(k-1)T)} < 1$ for any $t \in ((k-1)T, kT)$ and by Assumption \ref{hyp:Lambda}, we deduce that 
		\[
		X_s(kT) \leq e^{-\mu_s T} X_s((k-1)T) + T C_\Lambda,
		\] 
		with $C_\Lambda$ defined in \eqref{eqn:CLambda}. Using the iteration with respect to $k$, we deduce that 
		\[
		X_s(kT) \leq e^{-\mu_s kT} X_s^0 + T C_\Lambda \left(1 + e^{-\mu_s T} + \dots + e^{-\mu_s (k-1) T}\right) = e^{-\mu_s kT} X_s^0 + T C_\Lambda \frac{1-e^{-\mu_s kT}}{1-e^{-\mu_s T}}. 
		\]
		Now for any time $t > 0$, there exists an integer $k$ such that $t \in [kT, (k+1)T)$. Then, we obtain that 
		\[
		\begin{split}
			X_s(t) & = e^{-\mu_s(t - kT)} X_s(kT) + \int_{kT}^{t} e^{-\mu_s (t-s)} \Lambda(s) ds \\
			& \leq  e^{-\mu_st} X_s^0 + TC_\Lambda \frac{e^{-\mu_s (t - kT)} - e^{-\mu_s t}}{1 - e^{-\mu_s T}} + T C_\Lambda.  \\
			& \leq  e^{-\mu_st} X_s^0 +  \frac{TC_\Lambda}{1 - e^{-\mu_s T}} \left(1 - e^{-\mu_s t}\right) + T C_\Lambda \\ 
		\end{split}
		\]
		since $ e^{-\mu_s (t - kT)} < 1$. The inequality of $M_1^s + M_2^s$ follows. 
	\end{proof}
	
	\section{Mosquito dynamics without sterile males \label{sec:natural}}
	First, we describe the dynamics of wild mosquitoes in the two areas by considering the following system which is re-obtained from system \eqref{eqn:main2patch} in the absence of sterile males  
	\begin{subequations}
		\label{eqn:natural}
		\begin{align}
			\dot{E_1} &= bF_1\left( 1 - \frac{E_1}{K_1}\right) - (\nu_E + \mu_E) E_1, \label{eqn:E1n} \\
			\dot{F_1} &= r\nu_E E_1 - \mu_F F_1 - d_{12} F_1 + d_{21} F_2, \label{eqn:F1n} \\
			\dot{M_1} & = (1-r)\nu_E E_1 - \mu_M M_1 - \beta d_{12} M_1 + \beta d_{21} M_2, \label{eqn:M1n} \\
			\dot{E_2} &= bF_2\left( 1 - \frac{E_2}{K_2}\right) - (\nu_E + \mu_E) E_2, \label{eqn:E2n} \\
			\dot{F_2} &= r\nu_E E_2 - \mu_F F_2 - d_{21} F_2 + d_{12} F_1, \label{eqn:F2n}  \\
			\dot{M_2} & = (1-r)\nu_E E_2 - \mu_M M_2 - \beta d_{21} M_2 + \beta d_{12} M_1, \label{eqn:M2n}
		\end{align}
	\end{subequations}
	It is clear that the subset $\{0 \leq E_1 \leq K_1\} \cap \{0 \leq E_2 \leq K_2\}$ of the positive cone of $\mathbb{R}^6$ is positively invariant over time. 
	The following result shows the nature of the equilibrium points of system \eqref{eqn:natural}.   
	\begin{theorem}
		\label{thm:equi1}
		For $\mathcal{N} \leq 1$, zero is the unique equilibrium of system \eqref{eqn:natural}, and all trajectories of \eqref{eqn:natural} resulting from non-negative initial data converge to zero as time evolves.
		
		For $\mathcal{N} > 1$, system (\ref{eqn:natural}) has two equilibrium points: zero and $\mathbf{u}^+ = (E_1^+, F_1^+, M_1^+, E_2^+, F_2^+, M_2^+)$ strictly positive. Moreover, the zero equilibrium is unstable. All trajectories of \eqref{eqn:natural} resulting from any positive initial data $(E_1^0, F_1^0, M_1^0, E_2^0, F_2^0, M_2^0)$ such that $(E_1^0, F_1^0, E_2^0, F_2^0)  > 0$ converge to $\mathbf{u}^+$ when $t \rightarrow +\infty$. 
	\end{theorem}
	
	Theorem \ref{thm:equi1} shows that the constant $\mathcal{N}$ defined in \eqref{eqn:bon} is the basic offspring number of the whole two-patch system \eqref{eqn:main2patch}. When $\mathcal{N} > 1$, the populations in both areas remain persistent for any diffusion rates as time evolves. In the rest of the paper, we only consider the case $\mathcal{N} > 1$. 
	
	To prove this theorem, we first consider the sub-system of $E_1, \ E_2, \ F_1, \ F_2$. From equations (\ref{eqn:F1n}) and (\ref{eqn:F2n}), the positive equilibrium satisfies 
	\[
	r\nu_E\begin{pmatrix}
		E_1^+ \\ E_2^+
	\end{pmatrix} = \begin{pmatrix}
		\mu_F + d_{12} & -d_{21} \\
		-d_{12} & \mu_F + d_{21}
	\end{pmatrix} \begin{pmatrix}
		F_1^+ \\ F_2^+
	\end{pmatrix}, 
	\]
	then 
	\begin{equation}
		\begin{pmatrix}
			F_1^+ \\ F_2^+
		\end{pmatrix} = \frac{r\nu_E}{\mu_F(\mu_F + d_{12} + d_{21})}\begin{pmatrix}
			\mu_F + d_{21} & d_{21} \\
			d_{12} & \mu_F + d_{12}
		\end{pmatrix}\begin{pmatrix}
			E_1^+ \\ E_2^+
		\end{pmatrix}.
		\label{eqn:relation1}
	\end{equation}
	On the other hand, from equation (\ref{eqn:E1n}) and (\ref{eqn:E2n}), we also have 
	\begin{equation}
		F_1^+ = \frac{\nu_E + \mu_E}{b} \frac{E_1^+}{1 - \frac{E_1^+}{K_1}}, \qquad 	F_2^+ = \frac{\nu_E + \mu_E}{b} \frac{E_2^+}{1 - \frac{E_2^+}{K_2}}. 
		\label{eqn:relation2}
	\end{equation}
	From (\ref{eqn:relation1}) and (\ref{eqn:relation2}), we deduce that 
	\begin{equation}
		E_2^+ = \frac{\mu_F + d_{12} + d_{21}}{d_{21}}\frac{1}{\mathcal{N}} \frac{E_1^+}{1-\frac{E_1^+}{K_1}} - \frac{\mu_F + d_{21}}{d_{21}} E_1^+ =: f_{21}(E_1^+),
		\label{eqn:f21}
	\end{equation}
	\begin{equation}
		E_1^+ = \frac{\mu_F + d_{12} + d_{21}}{d_{12}}\frac{1}{\mathcal{N}} \frac{E_2^+}{1-\frac{E_2^+}{K_2}} - \frac{\mu_F + d_{12}}{d_{12}} E_2^+ =: f_{12}(E_2^+).
		\label{eqn:f12}
	\end{equation}
	The following lemma provides information for these functions. 
	\begin{lemma}
		For $\{i, j\} = \{1, 2\}$, function $f_{ij}(x)$ is defined and convex on $(0,K_j) \subset \mathbb{R}$. 
		
		If $\mathcal{N} \leq \dfrac{\mu_F + d_{ij} + d_{ji}}{\mu_F + d_{ij}}$, then $f_{ij}$ has no positive root and it is increasing on $(0,K_j)$. 
		
		\noindent Otherwise, it has a unique positive root 
		\begin{equation}
			\label{eqn:Kf}
			K_j^+ := K_j\left(1 - \dfrac{\mu_F + d_{ji} + d_{ij}}{\mu_F + d_{ij}} \dfrac{1}{\mathcal{N}}\right) < K_j
		\end{equation}
		Moreover, $f_{ij} < 0$ on $ (0,K_j^+)$, $f_{ij} > 0$ and increasing on $(K_j^+,K_j)$. 
		\label{lem:fij}
	\end{lemma}
	\begin{proof}[Proof of Lemma \ref{lem:fij}]
		We recall function $f_{ij}(x):= \dfrac{\mu_F + d_{ij} + d_{ji}}{d_{ij}}\dfrac{1}{\mathcal{N}} \dfrac{x}{1-\frac{x}{K_j}} - \dfrac{\mu_F + d_{ij}}{d_{ij}} x$. One has $f_{ij} = 0$ if and only if 
		\[
		\frac{\mu_F + d_{ij}}{K_j} x^2 + \left( \frac{\mu_F + d_{ij} + d_{ji}}{\mathcal{N}} - \mu_F - d_{ij} \right) x = 0.
		\]
		We deduce that $f_{ij} = 0$ at $0$ and $K_j^+$ as in \eqref{eqn:Kf}, and $K_j^+ > 0$ if and only if $\mathcal{N} > \dfrac{\mu_F + d_{ij} + d_{ji}}{\mu_F + d_{ij}}$. Moreover, $f_{ij} < 0$ on $ (0,K_j^+)$, $f_{ij} > 0$ and is increasing on $(K_j^+,K_j)$. It is defined and convex on $(0,K_j) \subset \mathbb{R}$ since 
		\[
		f_{ij}''(x) = 2\frac{\mu_F + d_{ij} + d_{ji}}{d_{ij}} \dfrac{1}{K_j\mathcal{N}} \frac{1}{\left(1 - \frac{x}{K_j}\right)^3} > 0,
		\]
		for any $x \in (0,K_j)$. We also have $f_{ij}(0) = 0$, $\lim\limits_{x \rightarrow K_j}f_{ij}(x) = +\infty$. 
	\end{proof}

	\begin{figure}
		\begin{subfigure}{0.45 \textwidth}
			\centering
			\includegraphics[width = \textwidth]{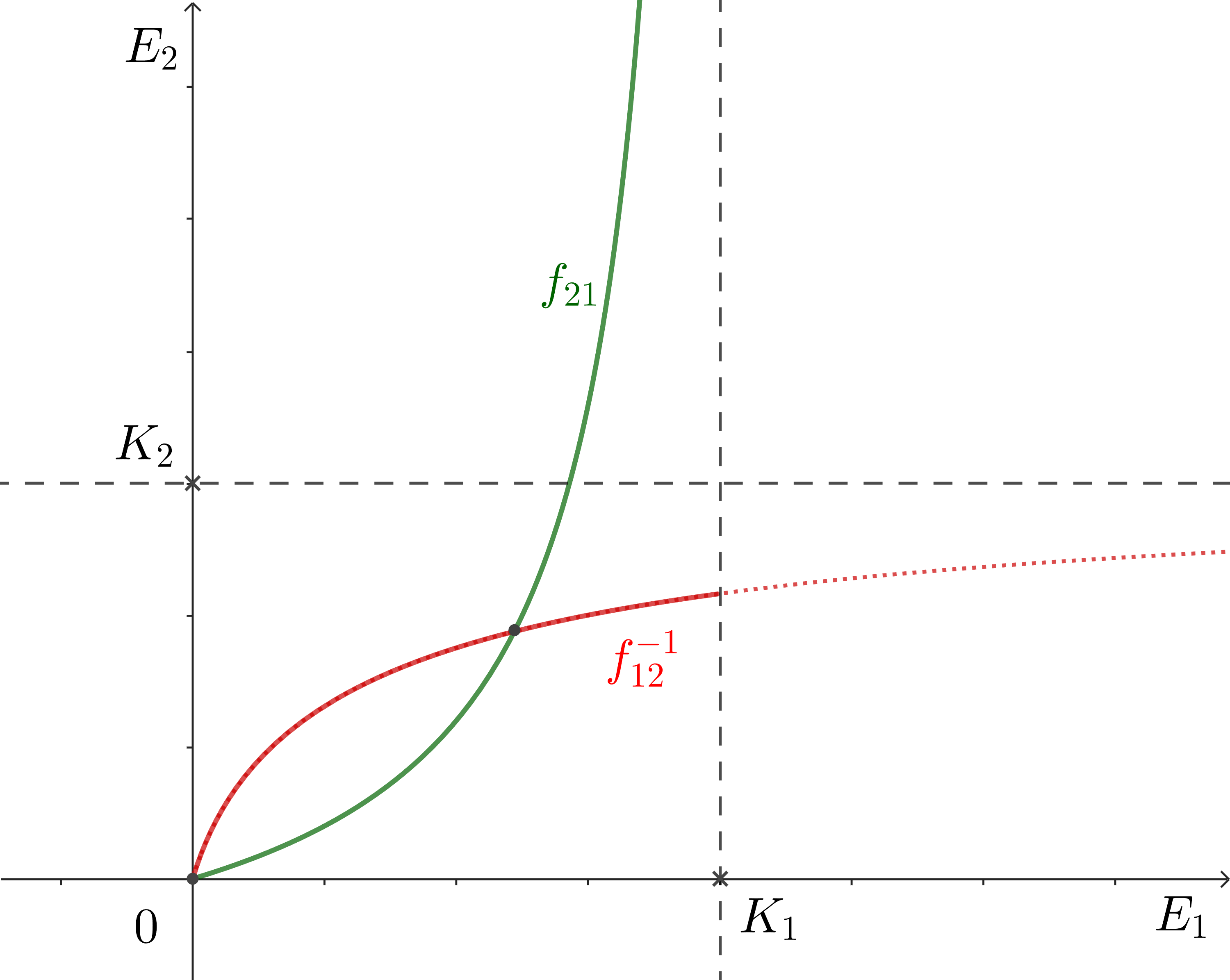}
			\subcaption{$1 < \mathcal{N} \leq \displaystyle \min_{\substack{i,j \in \{1, 2\}, i \neq j}} \dfrac{\mu_F + d_{ij} + d_{ji}}{\mu_F + d_{ij}}$}
		\end{subfigure}
		\hfill 
		\begin{subfigure}{0.45 \textwidth}
			\centering
			\includegraphics[width = \textwidth]{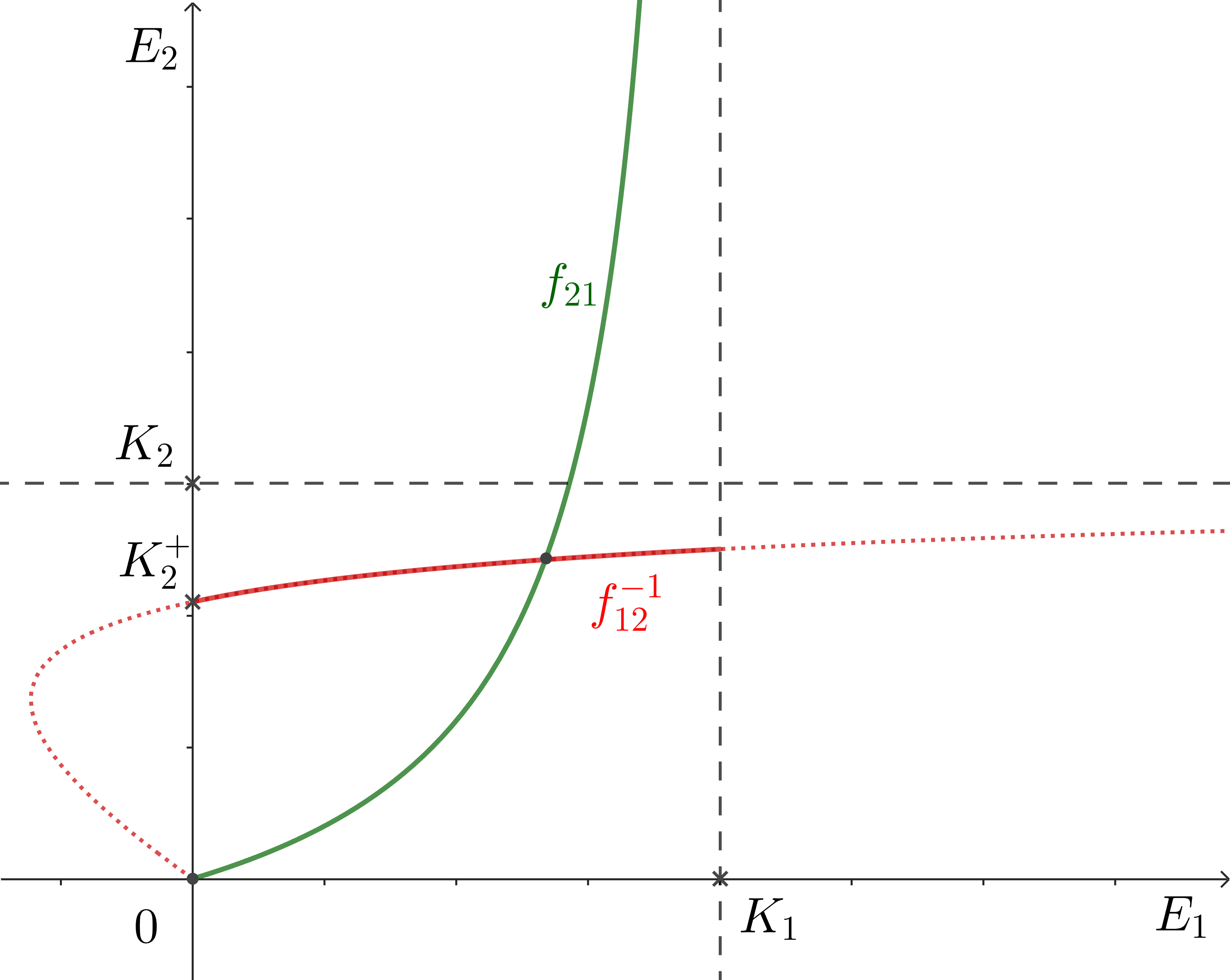}
			\subcaption{$\mathcal{N} > \dfrac{\mu_F + d_{12} + d_{21}}{\mu_F + d_{12}}$ with $d_{12} > d_{21}$}
		\end{subfigure}
		\caption{Relations $E_2^+ = f_{21}(E_1^+)$ (in green) and $E_1^+ = f_{12}(E_2^+)$ (in red) presented in the $E_1 - E_2$ plane when $\mathcal{N} > 1$. In case (a), function $f_{12}$ and $f_{21}$ are increasing and have a unique root zero, so $f_{12}$ is invertible on $[0,K_2)$ and the inverse function $f_{12}^{-1}: [0,K_1] \rightarrow [0,K_2)$ is plotted in the solid red line. In case (b), function $f_{12}$ has another positive root $K_2^+$ and is invertible on $[K_2^+,K_2)$. The inverse function $f_{12}^{-1}: [0,K_1] \rightarrow [K_2^+,K_2)$ is plotted in the solid red line.}
		\label{fig:fij}
	\end{figure}
	
	\begin{proof}[Proof of Theorem \ref{thm:equi1}]
		\textbf{Existence and uniqueness of the positive equilibrium.} 
		System (\ref{eqn:natural}) has a positive equilibrium iff system (\ref{eqn:f21})-(\ref{eqn:f12}) has a solution $(E_1^+, E_2^+)$ in $(0,K_1) \times (0,K_2)$. 
		
		\textbf{Case 1:} $0 < \mathcal{N} \leq \displaystyle \min_{\substack{i,j \in \{1, 2\} \\i \neq j}}\dfrac{\mu_F + d_{ij} + d_{ji}}{\mu_F + d_{ij}}$.
		
		According to Lemma \ref{lem:fij}, we have $f_{12}: [0, K_2) \rightarrow [0, +\infty)$ is positive and increasing, so this function is invertible (see Figure \ref{fig:fij}). We denote $f_{12}^{-1}: [0,K_1] \rightarrow [0, K_2)$ the restriction of the inverse function of $f_{12}$ on $[0,K_1]$, then 
		\[
		E_2^+ = f_{21}(E_1^+) = f_{12}^{-1}(E_1^+).
		\]
		Thus, $E_1^+$ is a positive root of function $f_{21} - f_{12}^{-1}$. For any $x \in (0,K_1)$, one has 
		\[
		(f_{21} - f_{12}^{-1})'(x) = f_{21}'(x) - \frac{1}{f_{12}'(f_{12}^{-1}(x))}, 
		\]
		then
		\[
		(f_{21} - f_{12}^{-1})''(x) = f_{21}''(x) + \frac{f_{12}''(f_{12}^{-1}(x))}{(f_{12}'(f_{12}^{-1}(x)))^3} > 0
		\]
		since $f_{ij}$ is convex on $(0,K_j)$. Hence, $f_{21} - f_{12}^{-1}$ is convex on $(0,K_1)$. Moreover,  we have $(f_{21} - f_{12}^{-1})(0) = 0$, and $\displaystyle \lim_{x \rightarrow K_1}(f_{21} - f_{12}^{-1})(x) = +\infty$. Therefore, this function has a unique positive root if and only if the derivative at zero is negative. We have
		\[
		(f_{21} - f_{12}^{-1})'(0) = \frac{\frac{1}{\mathcal{N}}(\mu_F + d_{12} + d_{21}) - \mu_F - d_{21}}{d_{21}} - \frac{d_{12}}{\frac{1}{\mathcal{N}}(\mu_F + d_{12} + d_{21}) - \mu_F - d_{12}}.
		\] 
		Since $\mathcal{N} \leq \displaystyle \min_{\substack{i,j \in \{1, 2\} \\i \neq j}}\dfrac{\mu_F + d_{ij} + d_{ji}}{\mu_F + d_{ij}} < \dfrac{\mu_F + d_{12}+ d_{21}}{\mu_F}$, then $(f_{21} - f_{12}^{-1})'(0) < 0$ if and only if $\mathcal{N} > 1$. 
		
		\textbf{Case 2:} $\mathcal{N} > \displaystyle \min_{\substack{i,j \in \{1, 2\} \\i \neq j}}\dfrac{\mu_F + d_{ij} + d_{ji}}{\mu_F + d_{ij}}$.
		
		Without loss of generality, we assume that $d_{12} > d_{21}$, then $\dfrac{\mu_F + d_{12} + d_{21}}{\mu_F + d_{12}} = \displaystyle \min_{\substack{i,j \in \{1, 2\} \\i \neq j}}\dfrac{\mu_F + d_{ij} + d_{ji}}{\mu_F + d_{ij}}$. 
		
		If $\mathcal{ N} > \dfrac{\mu_F + d_{12} + d_{21}}{\mu_F + d_{12}} > 1$, again according to Lemma \ref{lem:fij}, function $f_{12}$ has a unique positive root $K_2^+$ and is invertible on $[K_2^+, K_2]$ (see Figure \ref{fig:fij}(b)). We denote again $f_{12}^{-1}: [0, K_1] \rightarrow [K_2^+, K_2)$ the restriction of the  inverse function of $f_{12}$ on $[0,K_1]$, then we also have $f_{21} - f_{12}^{-1}$ convex on $(0,K_1)$, and $(f_{21} - f_{12}^{-1})(0) = -f_{12}^{-1}(0) = -K_2^+ < 0$, $\displaystyle \lim_{x \rightarrow K_1} (f_{21} - f_{12}^{-1})(x) = +\infty > 0$. We can deduce that $f_{21} - f_{12}^{-1}$ has a unique positive root on $(0,K_1)$. 
		
		\medskip \noindent {\bf Instability of the zero equilibrium. } At the origin $\mathbf{0}_6 = (0,0,0,0,0,0)$ of $\mathbb{R}^6$, the Jacobian matrix of system (\ref{eqn:natural}) is 
		\[J(\mathbf{0}) = \begin{pmatrix}
			-\nu_E - \mu_E & b & 0 & 0 & 0 & 0 \\
			r\nu_E & -\mu_F -d_{12} & 0 & 0 & d_{21} & 0 \\
			(1-r)\nu_E & 0 & -\mu_M - \beta d_{12} & 0 & 0 & \beta d_{21} \\
			0 & 0 & 0 & -\nu_E - \mu_E & b & 0 \\
			0 & d_{12} & 0 & r\nu_E & -\mu_F -d_{21} & 0 \\
			0 & 0 & \beta d_{12} & (1-r)\nu_E & 0 & -\mu_F - \beta d_{21}
		\end{pmatrix},\]
		with the characteristic polynomial 
		\[
		\det(J(0) - \lambda I) = \left[(\lambda + \mu_F)(\lambda + \mu_M) + \beta(d_{12} + d_{21} + d_{12}\mu_F + d_{21}\mu_M) \right]\left[(\lambda +\nu_E + \mu_E)(\lambda + \mu_F) - br\nu_E\right]
		\]
		\[
		\times \left[(\lambda +\nu_E + \mu_E)(\lambda + \mu_F + d_{12} + d_{21}) - br\nu_E\right]. 
		\]
		Since $\mathcal{N} > 1$, we have $\mu_F(\nu_E + \mu_E) - br\nu_E < 0$. Thus, we can deduce that the factor $(\lambda +\nu_E + \mu_E)(\lambda + \mu_F) - br\nu_E = \lambda^2 + \lambda(\nu_E + \mu_E + \mu_F) + \mu_F(\nu_E + \mu_E) - br\nu_E$ has one positive root $\lambda > 0$. Hence, the Jacobian at zero has at least one positive eigenvalue so the zero equilibrium is unstable. 
		
		\medskip \noindent {\bf Stability of the positive equilibrium. } First, we can see that the system (\ref{eqn:E1n})-(\ref{eqn:F1n}), (\ref{eqn:E2n})-(\ref{eqn:F2n}) of $(E_1, F_1, E_2, F_2)$ does not depend on $M_1, \ M_2$, and it is cooperative and irreducible. By applying Theorem 1.1 in Chapter 4 of \cite{SMI}, one deduces that this system is strongly monotone. When $\mathcal{N} > 1$, this system admits exactly two equilibria: $\mathbf{0}_4 = (0, 0, 0, 0)$, and $(E_1^+, F_1^+, E_2^+, F_2^+)$. But the zero equilibrium is unstable, so by Theorem 2.2 in Chapter 2 of \cite{SMI}, if the initial data satisfies that $\mathbf{0}_4 < (E_1^0, F_1^0, E_2^0, F_2^0) \leq (E_1^+, F_1^+, E_2^+, F_2^+)$, the solution $(E_1, F_1, E_2, F_2)$ converges to $(E_1^+, F_1^+, E_2^+, F_2^+)$ when $t \rightarrow +\infty$. 
		
		Now if the initial data satisfies that $(E_1^0, F_1^0, E_2^0, F_2^0) > (E_1^+, F_1^+, E_2^+, F_2^+)$, then there exists a constant $\lambda > 1$ large enough such that $\lambda (E_1^+, F_1^+, E_2^+, F_2^+) \geq (E_1^0, F_1^0, E_2^0, F_2^0)$. Since $1 - \dfrac{\lambda  E_i^+}{K_i} < 1 - \dfrac{E_i^+}{K_i}$, one has 
		\[
		b\lambda F_i^+\left( 1 - \frac{\lambda E_i^+}{K_i}\right) - (\nu_E + \mu_E) \lambda E_i^+ < \lambda \left[ bF_i^+\left( 1 - \frac{E_i^+}{K_i}\right) - (\nu_E + \mu_E) E_i^+ \right] = 0, 
		\]
		
		\noindent and the right-hand side of system (\ref{eqn:E1n})-(\ref{eqn:F1n}), (\ref{eqn:E2n})-(\ref{eqn:F2n}) at $\lambda (E_1^+, F_1^+, E_2^+, F_2^+) $ is non positive. Thus, the trajectory resulting from the initial data $\lambda (E_1^+, F_1^+, E_2^+, F_2^+)$ is non-increasing, and therefore converges to $(E_1^+, F_1^+, E_2^+, F_2^+)$. By applying the Lemma \ref{lem:chaplygin} to system (\ref{eqn:E1n})-(\ref{eqn:F1n}), (\ref{eqn:E2n})-(\ref{eqn:F2n}), we deduce that the trajectory resulting from the initial data $(E_1^0, F_1^0, E_2^0, F_2^0)$ lies between $(E_1^+, F_1^+, E_2^+, F_2^+)$ and the trajectories resulting from $\lambda(E_1^+, F_1^+, E_2^+, F_2^+)$. Hence, it also converges to $(E_1^+, F_1^+, E_2^+, F_2^+)$ when time $t$ goes to infinity. 
		
		Moreover, since the trajectories issued from the initial data above and below $(E_1^+, F_1^+, E_2^+, F_2^+)$ all converge to the same limit, then by the comparison principle, we deduce that the trajectory resulting from any positive initial data with values between these initial values converges to this equilibrium.

		Secondly, if we denote matrix $A = \begin{pmatrix}
			-\mu_M - \beta d_{12} & \beta d_{21} \\
			\beta d_{12} & -\mu_M - \beta d_{21}
		\end{pmatrix}$, this matrix is Hurwitz. Functions $M_1, \ M_2$ satisfy $\begin{pmatrix}
			\dot{M_1} \\ \dot{M_2}
		\end{pmatrix} = A \begin{pmatrix}
			M_1 \\ M_2
		\end{pmatrix} + (1-r)\nu_E \begin{pmatrix}
			E_1 \\ E_2
		\end{pmatrix}$. Thus, for any $t > 0$, 
		\begin{equation}
			\label{eqn:Mt}
			\begin{pmatrix}
				M_1(t) \\ M_2(t)
			\end{pmatrix} = e^{tA} \begin{pmatrix}
				M_1^0 \\ M_2^0 
			\end{pmatrix} + (1-r)\nu_E \displaystyle \int_0^t e^{(t-s)A} \begin{pmatrix}
				E_1(s) \\ E_2(s) 
			\end{pmatrix}ds.
		\end{equation}
		Moreover, the equilibrium satisfies
		\begin{equation}
			\label{eqn:M*}
			\begin{pmatrix}
				M_1^+ \\ M_2^+ 
			\end{pmatrix} = -(1-r)\nu_E A^{-1} \begin{pmatrix}
				E_1^+ \\ E_2^+
			\end{pmatrix}.
		\end{equation}
		Hence, from \eqref{eqn:M*} and \eqref{eqn:Mt}, we deduce that
		\begin{equation}
			\begin{pmatrix}
				M_1(t) \\ M_2(t)
			\end{pmatrix} - \begin{pmatrix}
				M_1^+ \\ M_2^+ 
			\end{pmatrix} = e^{tA} \begin{pmatrix}
				M_1^0 \\ M_2^0 
			\end{pmatrix} + (1-r)\nu_E \left[ \displaystyle \int_0^t e^{(t-s)A} \begin{pmatrix}
				E_1(s) \\ E_2(s) 
			\end{pmatrix}ds + A^{-1} \begin{pmatrix}
				E_1^+ \\ E_2^+
			\end{pmatrix}\right]. 
			\label{eqn:subtraction}
		\end{equation}
		
		Moreover, when $t \rightarrow +\infty$, we have that $\begin{pmatrix}
			E_1(t) \\ E_2(t) 
		\end{pmatrix}$ converges to $\begin{pmatrix}
			E_1^+ \\ E_2^+ 
		\end{pmatrix}$ and $e^{tA} \rightarrow 0$ since $A$ is Hurwitz. Thus, for any $\varepsilon > 0$, there exists a time $T_\varepsilon > 0$ large enough such that for any $t > T_\varepsilon$, 
		\begin{equation}
			\label{eqn:eps}
			E_i^+ - \varepsilon < E_i(t) < E_i^+ + \varepsilon, \qquad i = 1, 2,
		\end{equation}
		and $\lVert e^{tA}\rVert < \lVert e^{T_\varepsilon A} \rVert \leq  \varepsilon$.
		
		Since matrix $A$ is Metzler and irreducible, then by applying Lemma \ref{lem:metzler}, one has that $e^{At}$ is strictly positive for any  $t > 0$. Moreover, one has $E_i \in (0, K_i)$ in $(0, +\infty)$, then for any $t > 2T_\varepsilon$,
		\[
		\displaystyle 0 < \int_0^{T_\varepsilon} e^{(t-s)A} \begin{pmatrix}
			E_1(s) \\ E_2(s) 
		\end{pmatrix}ds < \int_0^{T_\varepsilon} e^{(t-s)A} ds \begin{pmatrix}
			K_1 \\ K_2 
		\end{pmatrix},
		\]
		then
		\[
		\displaystyle \norm{\int_0^{T_\varepsilon} e^{(t-s)A} \begin{pmatrix}
				E_1(s) \\ E_2(s) 
			\end{pmatrix}ds } < \norm{\int_0^{T_\varepsilon} e^{(t-s)A} ds} \norm{\begin{pmatrix}
				K_1 \\ K_2 
		\end{pmatrix}}  = \norm{A^{-1} \left(e^{tA} - e^{(t- T_\varepsilon) A}\right)} \norm{\begin{pmatrix}
				K_1 \\ K_2 
		\end{pmatrix}}  < \varepsilon C_1,
		\]
		with some positive constant $C_1$ not depending on $\varepsilon$. Using the second inequality in (\ref{eqn:eps}), one has 
		\[
		\displaystyle  \int_{T_\varepsilon}^t e^{(t-s)A} \begin{pmatrix}
			E_1(s) \\ E_2(s) 
		\end{pmatrix}ds \leq A^{-1} (e^{(t- T_\varepsilon)A} - I) \begin{pmatrix}
			E_1^+ + \varepsilon \\ E_2^+ + \varepsilon 
		\end{pmatrix}  = A^{-1} e^{(t- T_\varepsilon)A} \begin{pmatrix}
			E_1^+ + \varepsilon \\ E_2^+ + \varepsilon 
		\end{pmatrix} - \varepsilon A^{-1} -A^{-1} \begin{pmatrix}
			E_1^+ \\ E_2^+ 
		\end{pmatrix} . 
		\]
		Proving similarly for the other inequality, we can deduce that 
		\[
		\displaystyle\norm{\int_{T_\varepsilon}^t e^{(t-s)A} \begin{pmatrix}
				E_1(s) \\ E_2(s) 
			\end{pmatrix}ds  +  A^{-1} \begin{pmatrix}
				E_1^+ \\ E_2^+ 
		\end{pmatrix}} \leq \varepsilon C_2,
		\]
		with some positive constant $C_2$ not depending on $\varepsilon$. Hence, from \eqref{eqn:subtraction}, we deduce that for any $t > 2T_\varepsilon$, 
		\[
		\displaystyle \norm{ \begin{pmatrix}
				M_1(t) \\ M_2(t) 
			\end{pmatrix}  - \begin{pmatrix}
				M_1^+ \\ M_2^+ 
		\end{pmatrix}} <  \varepsilon \left[ \norm{\begin{pmatrix}
				M_1^0 \\ M_2^0 
		\end{pmatrix}} + C_1 + C_2\right]. 
		\]
		Therefore, $(M_1(t), M_2(t))$ converges to $(M_1^+, M_2^+)$ when $t$ tends to $+\infty$. 
	\end{proof}
	
	In the following section, by considering the releases of sterile males, we look for a condition of release functions $\Lambda$ such that the positive equilibrium disappears. 
	\section{Elimination with releases of sterile males \label{sec:elimination}}
	In this section, we consider $\Lambda(t)$ in system \eqref{eqn:main2patch} the number of sterile males released per time unit, and our goal is to adjust its values such that the wild population reaches elimination. We consider two release strategies as follows
	
	\medskip \noindent \textbf{Constant release:} Let the release function $\Lambda(t) \equiv \Lambda > 0$. As time goes to infinity, the density of sterile males $(M_1^{s},M_2^{s})$ converges to $(M_1^{s*},M_2^{s*})$ that is the solution of system 
	\begin{equation}
		\label{eqn:Msbound}
		\begin{cases}
			\Lambda - \mu_s M_1^{s*} - \alpha d_{12} M_1^{s*} + \alpha d_{21} M_2^{s*} = 0 \\
			- \mu_s M_2^{s*} - \alpha d_{21} M_2^{s*} + \alpha d_{12} M_1^{s*} = 0
		\end{cases}
	\end{equation}
	By denoting 
	\begin{equation}
		\label{eqn:ki}
		\tau_1 := \dfrac{(\mu_s + \alpha d_{21})}{\mu_s(\mu_s + \alpha d_{12} + \alpha d_{21})}, \quad \tau_2 := \dfrac{\alpha d_{12}}{\mu_s(\mu_s + \alpha d_{12} + \alpha d_{21})}, 
	\end{equation}
	we have $M_1^{s*}  = \tau_1 \Lambda$, and $M_2^{s*} =  \tau_2 \Lambda $ and  $M_1^{s*} + M_2^{s*} = \dfrac{\Lambda}{\mu_s}$. 
	
	\medskip \noindent \textbf{Impulsive periodic releases:}  Consider the release function 
	\begin{equation}
		\label{eqn:Lambdaper}
		\Lambda(t) = \displaystyle \sum_{k = 0}^{+\infty} \tau \Lambda^\mathrm{per}_k \delta_{k \tau},
	\end{equation}
	with period $\tau > 0$ and $\Lambda^\mathrm{per}_k$ is the average number of sterile males released per time unit during the time interval $(k\tau, (k+1)\tau)$ for $k = 0, 1, \dots$. We choose in this work $\Lambda^\mathrm{per}_k$ constant and drop consequently the sub-index $k$. The release function $\Lambda(t)$ in \eqref{eqn:Lambdaper} means that we release a total amount of $\tau \Lambda^\mathrm{per}$ mosquitoes at the beginning of each time period ($t = k\tau$).
	
	Denote vector $X(t) = \begin{pmatrix} M_1^s(t) \\ M_2^s(t) \end{pmatrix}$, then with $k = 0, 1, \dots$, the density of sterile males satisfies the following system
	\begin{subequations}
		\label{eqn:periodic}
		\begin{align}
			X'(t) & = A_s X(t) \qquad \text{for any } t \in \displaystyle \bigcup_{k= 0}^{\infty} \left(k\tau, (k+1)\tau\right), \label{eqn:per1}\\
			X(k\tau^+) & = X(k\tau^-) + \begin{pmatrix}
				\tau \Lambda^\mathrm{per} \\ 0
			\end{pmatrix}, \label{eqn:per2}
		\end{align}
	\end{subequations}
	with matrix $A_s = \begin{pmatrix}
		-\alpha d_{12} - \mu_s & \alpha d_{21} \\
		\alpha d_{12} & -\alpha d_{21} - \mu_s
	\end{pmatrix}$, and $X(k\tau^\pm)$ denote the right and left limits of $X(t)$ at time $k\tau$ and by convention, we set $X(0^-) = 0$. 
	
	The following lemma shows that $X(t)$ converges to a periodic solution as time evolves and this periodic solution is discontinuous and bounded from below by a positive value. 
	\begin{lemma}
		\label{lem:periodic}
		The solution $X(t)$ of system \eqref{eqn:periodic} converges to a periodic solution $X^\mathrm{per}(t)$ that satisfies
			\begin{equation}
			\label{eqn:Xper}
			X^\mathrm{per}(k\tau^+) = (I - e^{A_s\tau})^{-1} \begin{pmatrix}
			\tau \Lambda^\mathrm{per} \\ 0
			\end{pmatrix}, \qquad X^\mathrm{per}(t) = e^{A_st} X^\mathrm{per}(k\tau^+) \quad \text{ for any } t \in \displaystyle \bigcup_{k= 0}^{\infty} \left(k\tau, (k+1)\tau\right).
			\end{equation}
		Moreover, there exist positive constants $\tau_1^\mathrm{per}, \ \tau_2^\mathrm{per}$ which depend on $A_s$ and period $\tau$ such that for any $t > 0$, one has $ X^\mathrm{per}(t) \geq  \Lambda^\mathrm{per} \begin{pmatrix}
			\tau_1^\mathrm{per} \\ \tau_2^\mathrm{per}
		\end{pmatrix}$. 
	\end{lemma}
	\begin{proof}
		The densities of the sterile males evolve according to \eqref{eqn:per1} on the union of open intervals $(k \tau, (k+1)\tau)$ while $X$ is submitted to jump at each point $k\tau$ as in \eqref{eqn:per2}. For such a release schedule, the solution of system \eqref{eqn:periodic} satisfies
		\[
		X(k\tau^+) = \displaystyle \sum_{i = 0}^{k} e^{iA_s\tau} \begin{pmatrix}
		\tau \Lambda^\mathrm{per} \\ 0
		\end{pmatrix}, \qquad X(t) = e^{A_st} X(k\tau^+) \quad \text{ for any } t \in \displaystyle \bigcup_{k= 0}^{\infty} \left(k\tau, (k+1)\tau\right).  
		\]
		Since matrix $A_s$ is Hurwitz, when $t \rightarrow +\infty$, we have that $X$ converges to the periodic solution $X^\mathrm{per} = \begin{pmatrix}
		M_1^{s, \mathrm{per}} \\ M_2^{s, \mathrm{per}}
		\end{pmatrix}$ that satisfies that \eqref{eqn:Xper} for  $k = 0, 1, \dots$
		
		We have matrix $A_s$ is Metzler and irreducible, then by applying Lemma \ref{lem:metzler}, we deduce that $e^{A_s} \gg 0$. 
		
		On the other hand, matrix $A_s$ is Hurwitz so $(I - e^{A_s\tau})^{-1} = \displaystyle \sum_{i = 0}^{+\infty} e^{i A_s\tau} \gg 0$ for any $\tau > 0$. Moreover, we have $e^{A_st}$ is also strictly positive for any $t > 0$, so
		\[
		\displaystyle \inf_{t > 0} X^\mathrm{per}(t) = \min_{t \in [0, \tau]} X^\mathrm{per}(t) = \min_{t \in [0, \tau]} e^{A_s t}(I - e^{A_s\tau})^{-1} \begin{pmatrix}
			\tau \Lambda^\mathrm{per} \\ 0
		\end{pmatrix}. 
		\]
		By taking  $\begin{pmatrix}
			\tau_1^\mathrm{per} \\ \tau_2^\mathrm{per}
		\end{pmatrix} = \displaystyle \min_{t \in [0, \tau]} e^{A_s t}(I - e^{A_s\tau})^{-1} \begin{pmatrix}
			\tau \\ 0
		\end{pmatrix}$, the result of Lemma \ref{lem:periodic} follows. 
	\end{proof}
	\begin{remark}
		The parameters $\tau_i$ defined in \eqref{eqn:ki} and $\tau_i^\mathrm{per}$ play a similar role to each other: they define a relationship between an average release rate $\Lambda$ per time unit and a (minimum) level of the sterile mosquito density. These parameters depend on the diffusion rates, the death rate of sterile males, and the release period $\tau$ in the periodic case.
	\end{remark}
	
	We provide in the following result a condition on $\Lambda(t)$ for the wild population to reach elimination.
	\begin{theorem}
		\label{thm:main}
		Consider system \eqref{eqn:main2patch} with the release function $\Lambda(t)$. Then 
		\begin{itemize}[leftmargin = 0.5cm]
			\item In the constant release case, for $\Lambda(t) \equiv \Lambda$, there exists a positive number  $\overline{\Lambda}$ satisfying 
			\[ \overline{\Lambda} \leq \displaystyle \max_{i = 1, 2} \dfrac{1}{\gamma \tau_i}(\mathcal{N} - 1)C_M \qquad i = 1, 2,  \]
			with $\tau_i$ defined in \eqref{eqn:ki}, $C_M$ defined in Lemma \ref{lem:bound} such that if $\Lambda > \overline{\Lambda}$, system (\ref{eqn:main2patch}) has a unique equilibrium 
			\[
			\mathbf{u}_0^*= (0,0,0, M_1^{s*}, 0, 0, 0, M_2^{s*}).
			\]
			Moreover, in this case, for any non-negative initial data, the solution of \eqref{eqn:main2patch} converges to this equilibrium when $t \rightarrow +\infty$.
			\item In the periodic release case, for $\Lambda(t)$ defined in \eqref{eqn:periodic}, There exists a positive constant $\overline{\Lambda}^\mathrm{per}$ satisfying 
			\[ \overline{\Lambda}^\mathrm{per} \leq \displaystyle \max_{i = 1, 2} \dfrac{1}{\gamma \tau_i^\mathrm{per}}(\mathcal{N} - 1)C_M,\] 
			with  $\tau_i^\mathrm{per}$ defined in Lemma \ref{lem:periodic},  $C_M$ defined in Lemma \ref{lem:bound}  such that  if  $\Lambda^\mathrm{per} > \overline{\Lambda}^\mathrm{per}$, then for any non-negative initial data, the solution of the initial value problem of system \eqref{eqn:main2patch} converges to the unique steady state
			\[
			\mathbf{u}_0^\mathrm{per} = (0,0,0,M_1^{s,\mathrm{per}},0,0,0,M_2^{s,\mathrm{per}}),
			\]
			as time $t \rightarrow +\infty$ .
		\end{itemize}
	\end{theorem}
	This result shows that with a sufficiently large number of sterile males released in the first zone, we can succeed in driving the wild population in both areas to elimination. In the following, we describe the principle idea to prove this result.

	\subsection{Principle of the method}
	\label{subsec:principle}
	To provide conditions for the release $\Lambda$ to stabilize the zero equilibrium, our strategy is as follows: 
	
	Step 1:  We consider $\rho_i = \displaystyle \sup_{t > 0} \dfrac{M_i(t)}{M_i(t) + \gamma M_i^s(t)}$, for $i = 1, \ 2$, in system \eqref{eqn:main2patch} to be smaller than some level, then we study the system with the fractions replaced by some constant.  
	
	Step 2: We show how to realize, through an adequate choice of $\Lambda$, the above behavior of $M_i^s$. 
	\subsubsection{Step 1: Setting the sterile population level directly}
	Theorem \ref{thm:equi1} shows us that when the basic offspring number is smaller than 1, the zero equilibrium is globally asymptotically stable. For the controlled system, the basic offspring numbers are smaller than $\rho_i \mathcal{N}$. It suggests that for stabilizing the origin of system \eqref{eqn:main2patch}, it is sufficient to ensure $\rho_i \mathcal{N} \leq 1$.
	\begin{proposition}
		\label{prop:rho}
		If the trajectory resulting from any positive initial data of system \eqref{eqn:main2patch} satisfies that for $\mathcal{N}$ defined in \eqref{eqn:bon},
		\begin{equation}
			\label{eqn:rho}
			\dfrac{M_i(t)}{M_i(t) + \gamma M_i^s(t)} \leq \frac{1}{\mathcal{N}}, \qquad t \geq 0, \ i = 1, 2. 
		\end{equation}
		then $\mathbf{u}' = (E_1, F_1, M_1, E_2, F_2, M_2)$ converges to $\mathbf{0}_6$ as time $t \rightarrow +\infty$. 
	\end{proposition}
	\begin{proof}
		Assume that we can set $M_i^s$ to be large enough such that \eqref{eqn:rho} holds, and we consider the following system 
		\begin{subequations}
			\label{eqn:sysrho}
			\begin{align}
				\dot{E_1} &= bF_1\left( 1 - \frac{E_1}{K_1}\right) - (\nu_E + \mu_E) E_1, \label{eqn:E1rho} \\
				\dot{F_1} &= \dfrac{1}{\mathcal{N}} r\nu_E E_1 - \mu_F F_1 - d_{12} F_1 + d_{21} F_2, \label{eqn:F1rho} \\
				\dot{M_1} & = (1-r)\nu_E E_1 - \mu_M M_1 - \beta d_{12} M_1 + \beta d_{21} M_2, \label{eqn:M1rho} \\
				\dot{E_2} &= bF_2\left( 1 - \frac{E_2}{K_2}\right) - (\nu_E + \mu_E) E_2, \label{eqn:E2rho} \\
				\dot{F_2} &= \dfrac{1}{\mathcal{N}} r\nu_E E_2 - \mu_F F_2 - d_{21} F_2 + d_{12} F_1, \label{eqn:F2rho}  \\
				\dot{M_2} & = (1-r)\nu_E E_2 - \mu_M M_2 - \beta d_{21} M_2 + \beta d_{12} M_1, \label{eqn:M2rho}
			\end{align}
		\end{subequations}
		Denote $\widetilde{\mathbf{u}} = (\widetilde{E_1}, \widetilde{F_1}, \widetilde{M_1}, \widetilde{E_2}, \widetilde{F_2}, \widetilde{M_2})$ solution of system \eqref{eqn:sysrho}. Since system \eqref{eqn:sysrho} is cooperative and the inequality \eqref{eqn:rho} holds, one obtains that $\widetilde{\mathbf{u}}$ is a super-solution of the system \eqref{eqn:E1}-\eqref{eqn:M1}, \eqref{eqn:E2}-\eqref{eqn:M2}, and by applying Lemma \ref{lem:chaplygin}, we have $\widetilde{\mathbf{u}} \geq \mathbf{u}'$. 
		
		Denote $\mathbf{u}^* = (E_1^*, F_1^*, M_1^*, E_2^*, F_2^*, M_2^*)$ a positive equilibrium of system \eqref{eqn:sysrho} if exists. Similar to the previous section, we have 
		\begin{equation}
			E_2^* = \frac{\mu_F + d_{12} + d_{21}}{d_{21}} \frac{E_1^*}{1-\frac{E_1^*}{K_1}} - \frac{\mu_F + d_{21}}{d_{21}} E_1^* =: g_{21}(E_1^*),
			\label{eqn:g21}
		\end{equation}
		\begin{equation}
			E_1^* = \frac{\mu_F + d_{12} + d_{21}}{ d_{12}} \frac{E_2^*}{1-\frac{E_2^*}{K_2}} - \frac{ \mu_F + d_{12}}{ d_{12}} E_2^* =: g_{12}(E_2^*).
			\label{eqn:g12}
		\end{equation}
		The analysis of $g_{ij}$ is analogous to $f_{ij}$ in Lemma \ref{lem:fij}. It is easy to check that $g_{12}$ is increasing on $(0,K_2)$, so it is invertible. Then, $E_1^*$ satisfies $(g_{21}-g_{12}^{-1})(E_1^*) = 0$. Function $g_{21}$ is convex and $g_{12}^{-1}$ is concave in $(0,K_1)$, and 
		\[
		g_{21}'(0) = \frac{\mu_F + d_{12} + d_{21}}{ d_{21}} - \frac{(\mu_F + d_{21})}{ d_{21}} = \frac{d_{12}}{d_{21}}, 
		\]
		\[
		(g_{12}^{-1})'(0) = \frac{1}{g_{12}'(0)} = \frac{1}{\frac{\mu_F + d_{12} + d_{21}}{d_{12}} - \frac{\mu_F + d_{12}}{ d_{12}}} = \frac{d_{12}}{d_{21}}.
		\]
		We obtain that $g'_{21}(0) = (g_{12}^{-1})'(0)$ (see Figure \ref{fig:gij}), so zero is the unique equilibrium of system \eqref{eqn:sysrho}. By applying Theorem 3.1 in Chapter 2 of \cite{SMI}, we deduce that when $t \rightarrow +\infty$, the solution $\widetilde{\mathbf{u}}(t)$ converges to the equilibrium zero. Since $\mathbf{u}'(t) \leq \widetilde{\mathbf{u}}(t)$ for all $t > 0$, we deduce that the $\mathbf{u}'$ also converges to zero when $t$ large.
		\begin{figure}
			\centering
			\includegraphics[width = 0.45 \textwidth]{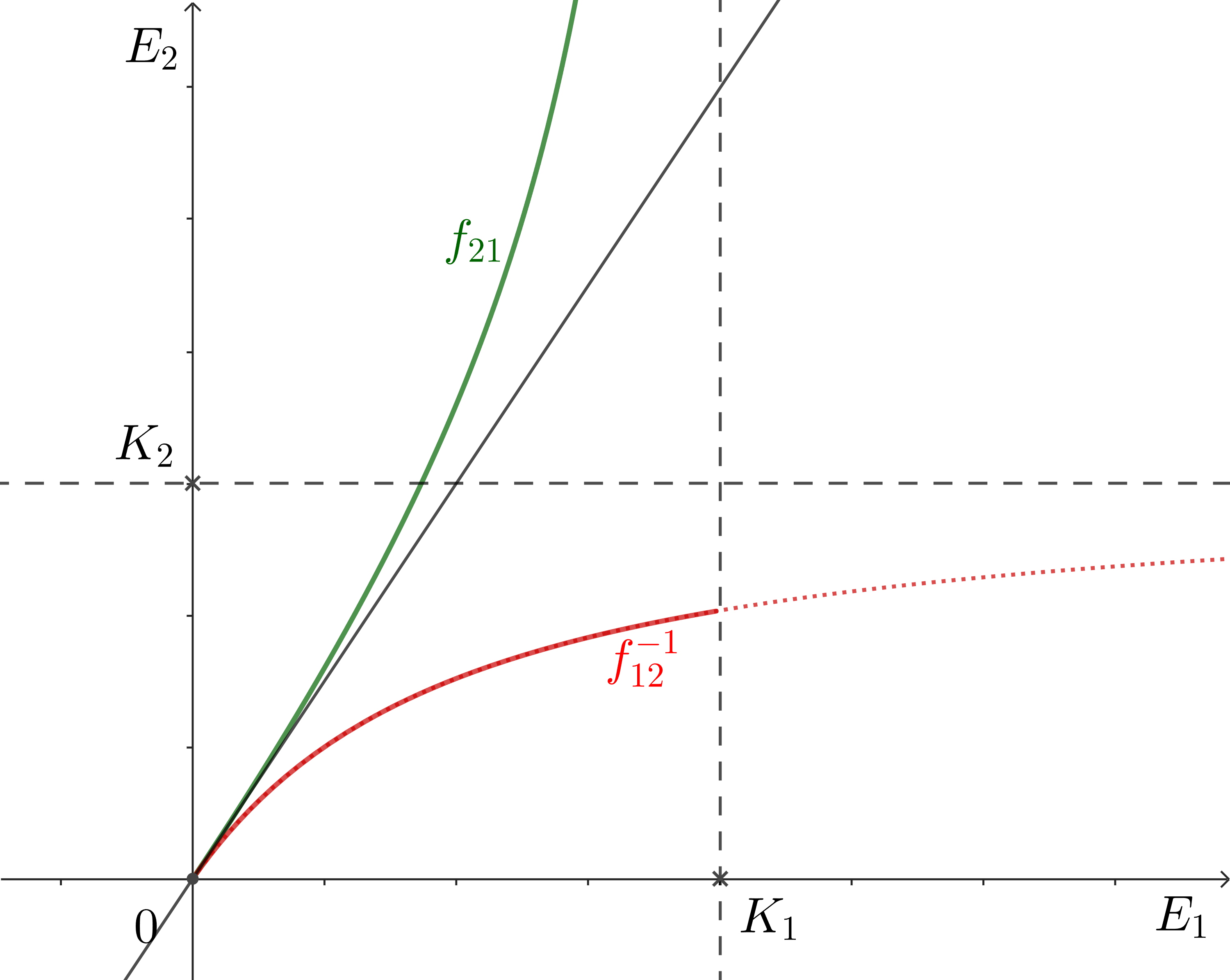}
			\caption{Relations $E_2^* = g_{21}(E_1^*)$ (in green) and $E_1^* = g_{12}(E_2^*)$ (in red) presented in the $E_1 - E_2$ plane when $\mathcal{N} > 1$. The two curves only intersect at the origin.}
			\label{fig:gij}
		\end{figure}
		
	\end{proof}
	\subsubsection{Step 2: Shaping the release function}
	We now want to choose $\Lambda$ such that the condition \eqref{eqn:rho} holds, which means 
	\[
	\gamma M_i^s(t) \geq \left(\mathcal{N} - 1\right) M_i(t), \qquad t \geq 0, \ i = 1, \ 2
	\]
	The upper bound of $M_i$ can be obtained from Lemma \ref{lem:bound}. For time $t > 0$ large enough, it is sufficient to choose $\Lambda$ such that $\gamma M_i^s(t) \geq \left(\mathcal{N} - 1\right) C_M$. 
	
	\subsection{Proof of Theorem \ref{thm:main}}
	\label{subsec:prooftwopatch}
	\begin{proof}[Proof of Theorem \ref{thm:main}] 
		\textbf{Constant release:} For $\Lambda(t) \equiv \Lambda$, let us recall $A_s = \begin{pmatrix}
			-\mu_s - \alpha d_{12} & \alpha d_{21} \\
			\alpha d_{12} & -\mu_s - \alpha d_{21}
		\end{pmatrix}$, and it is a Hurwitz matrix. And we have 
		\[ \begin{pmatrix}
			M_1^s(t) \\ M_2^s(t)
		\end{pmatrix} = e^{tA_s} \begin{pmatrix}
			M_1^{s0} \\ M_2^{s0} 
		\end{pmatrix} + (1-r)\nu_E \displaystyle \int_0^t e^{(t-s)A_s}ds \begin{pmatrix}
			\Lambda \\ 0 
		\end{pmatrix} .
		\]
		Thus, when $t \rightarrow +\infty$, we can deduce that $\begin{pmatrix}
			M_1^s(t) \\ M_2^s(t)
		\end{pmatrix}$ converges to $\begin{pmatrix}
			M_1^{s*} \\ M_2^{s*}
		\end{pmatrix}$ for any initial data since $e^{tA_s} \rightarrow 0$ when $t\rightarrow +\infty$.  Hence, for any $\varepsilon > 0$, there exists a value $T_\varepsilon > 0$ such that for any $t > T_\varepsilon$, $i = 1,2$,  
		\[
		M_i^s(t) \geq M_i^{s*} - \varepsilon. 
		\]
		If we take $\Lambda$ such that $\gamma(M_i^{s*} - \varepsilon) \geq \left(\mathcal{N} - 1\right) C_M$ with $C_M$ defined in Lemma \ref{lem:bound}, then condition \eqref{eqn:rho} holds. By applying Proposition \ref{prop:rho}, we deduce that for $i = 1, 2$, if $\Lambda > \displaystyle \max_{i = 1, 2} \dfrac{1}{\gamma \tau_i}(\mathcal{N} - 1)C_M$, system \eqref{eqn:main2patch} has a unique equilibrium point
		\[
		\mathbf{u}_0^*= (0,0,0, M_1^{s*}, 0, 0, 0, M_2^{s*}),
		\]
		and every trajectory converges to this equilibrium when $t \rightarrow +\infty$. 
		The dynamics of system \eqref{eqn:main2patch} depend continuously and monotonically on $\Lambda$, then we deduce that there exists a positive critical value $\overline{\Lambda} \leq \displaystyle \max_{i = 1, 2} \dfrac{1}{\gamma \tau_i}(\mathcal{N} - 1) C_M$ such that for any $\Lambda > \overline{\Lambda}$, and for any non-negative initial data, solution $\mathbf{u}' = (E_1, F_1, M_1, E_2, F_2, M_2)(t)$ converges to $\mathbf{0}_6$ when $t \rightarrow +\infty$.
		
		\medskip \noindent \textbf{Impulsive periodic releases:} Consider $\Lambda(t)$ defined in \eqref{eqn:periodic}, denote $(E_1^\mathrm{per}, F_1^\mathrm{per}, M_1^\mathrm{per}, E_2^\mathrm{per}, F_2^\mathrm{per}, M_2^\mathrm{per})$ a solution of \eqref{eqn:E1}-\eqref{eqn:M1}, \eqref{eqn:E2}-\eqref{eqn:M2} with $M_i^s \equiv M_i^{s, \mathrm{per}}$ defined in \eqref{eqn:Xper}.  From  Lemma \ref{lem:periodic}, one has $M_i^{s, \mathrm{per}} (t) \geq \Lambda^\mathrm{per} \tau_i^\mathrm{per}$ for all $t > 0$. Therefore, if we take $\Lambda^\mathrm{per}$ such that $\gamma \Lambda^\mathrm{per} \tau_i^\mathrm{per} \geq \left(\mathcal{N} - 1\right) C_M$, then condition \eqref{eqn:rho} holds. By applying Proposition \ref{prop:rho}, we deduce that $\mathbf{u}^\mathrm{per} = (E_1^\mathrm{per}, F_1^\mathrm{per}, M_1^\mathrm{per}, E_2^\mathrm{per}, F_2^\mathrm{per}, M_2^\mathrm{per})$ converges to zero as $t$ grows. 
		Since $(M_1^s, M_2^s)$ converges to $(M_1^{s, \mathrm{per}}, M_2^{s, \mathrm{per}})$ as $t \rightarrow +\infty$, we have $\mathbf{u}'=(E_1, F_1, M_1, E_2, F_2, M_2)$ approaches $\mathbf{u}^\mathrm{per}$ and thus converges to $\mathbf{0}_6$ as time $t$ goes to infinity. 
		
		Since the dynamics of system \eqref{eqn:main2patch} depends continuously and monotonically on $\Lambda$, there exists a positive critical value $\overline{\Lambda}^\mathrm{per} \leq \displaystyle \max_{i = 1, 2} \dfrac{1}{\gamma \tau_i^\mathrm{per}}(\mathcal{N} - 1) C_M $ such that if $\Lambda^\mathrm{per} > \overline{\Lambda}^\mathrm{per}$ the equilibrium $\mathbf{u}_0^\mathrm{per}$ of \eqref{eqn:main2patch} with $\Lambda(t)$ defined in \eqref{eqn:Lambdaper} is globally asymptotically stable.  
	\end{proof}
	\section{Parameter dependence of the critical values of the release rate \label{sec:dependence}}
	In this section, we consider the constant release case and examine how the critical value $\overline{\Lambda}$ depends on the parameters of system \eqref{eqn:main2patch}. In this model, the elimination of the population depends not only on the diffusion rate between the inaccessible area and the treated area, but also on the biological intrinsic values like the birth/death rates, and the carrying capacity. 
	
	\subsection{Diffusion rates}
	In this part, we want to compare the critical values of $\Lambda$ corresponding to different values of $d_{12}, \ d_{21}$. We show that when the diffusion rates are large enough, the critical number of sterile males released is the same as in the case when there is no separation between the two sub-populations. 
	
	\subsubsection{The case $d_{12}, \ d_{21}$ large}
	\label{subsec:largediffusion}
	First, we present a result of uniform convergence of system \eqref{eqn:main2patch} when $d_{12}, \ d_{21}$ go to $+\infty$ and $d_{12}$ is proportional to $d_{21}$. 
	\begin{proposition}
		\label{prop:infty}
		For $\varepsilon > 0$, consider the diffusion rates $d_{12}  = \dfrac{1}{\varepsilon}, \ d_{21} = \dfrac{\eta}{\varepsilon}$ with $\eta = \dfrac{d_{21}}{d_{12}} > 0$. Denote $\mathbf{u}^\varepsilon = (E_1^\varepsilon, F_1^\varepsilon, M_1^\varepsilon, M_1^{s, \varepsilon}, E_2^\varepsilon, F_2^\varepsilon, M_2^\varepsilon, M_2^{s, \varepsilon})$ the solution of system \eqref{eqn:main2patch} with the initial data $\mathbf{u}^{\varepsilon,0}$ satisfying that 
		\[
		\{E^{\varepsilon, 0}_i\}, \ \{F^{\varepsilon, 0}_i\}, \ \{M^{\varepsilon, 0}_i\} \text{  converge to } E^{0, 0}_i, F^{0, 0}_i, M^{0, 0}_i \text{ respectively as } \varepsilon \rightarrow 0, \text{ with } i = 1, \ 2, 
		\]
		and 
		\begin{equation}
			\label{eqn:initial}
			E^{\varepsilon, 0}_1 - \eta E^{\varepsilon, 0}_2 = \mathcal{O}(\varepsilon), \quad F^{\varepsilon, 0}_1 - \eta F^{\varepsilon, 0}_2 = \mathcal{O}(\varepsilon), \quad M^{\varepsilon, 0}_1 - \eta M^{\varepsilon, 0}_2 = \mathcal{O}(\varepsilon), \quad M^{s, \varepsilon, 0}_1 = M^{s, \varepsilon, 0}_2 = 0.
		\end{equation}
		Then, when $\varepsilon \rightarrow 0$, the sequence $\{\mathbf{u}^\varepsilon\}$ converges uniformly to a limit $(E_1, F_1, M_1, M^s_1, E_2, F_2, M_2, M^s_2)$ on $[0, +\infty)$. Moreover, we have
		\begin{equation}
			\label{eqn:limit}
			F_1 = \eta F_2, \quad M_1 = \eta M_2, \quad M_1^s = \eta M_2^s. 
		\end{equation}
		If we denote $F = F_1 + F_2,\ M = M_1 + M_2, \ M^s = M^s_1 + M^s_2$ , then $(E_1, E_2, F, M, M^s)$ solves the following system 
		\begin{subequations}
			\label{eqn:infinity}
			\begin{align}
				\dot{E_1} &= \frac{\eta}{\eta + 1} bF\left( 1 - \frac{E_1}{K_1}\right) - (\nu_E + \mu_E) E_1, \label{eqn:E1inf} \\
				\dot{E_2} &= \frac{1}{\eta + 1}bF\left( 1 - \frac{E_2}{K_2}\right) - (\nu_E + \mu_E) E_2, \label{eqn:E2inf} \\
				\dot{F} &= r\nu_E (E_1+E_2) \frac{M}{M + \gamma M^s}- \mu_F F, \label{eqn:Finf} \\
				\dot{M} & = (1-r)\nu_E (E_1+E_2) - \mu_M M , \label{eqn:Minf} \\
				\dot{M^s} & = \Lambda_\infty - \mu_s M^s,  \label{eqn:Msinf}
			\end{align}
		\end{subequations}
		with the corresponding initial data $E^{0,0}_1, \ E^{0, 0}_2, \ F^{0,0} = F^{0, 0}_1 + F^{0, 0}_2, \ M^{0,0} = M^{0, 0}_1 + M^{0, 0}_2, \ M^{s,0,0} = 0$.
	\end{proposition}
	\medskip 
	It is straightforward to see that the previous result implies that $F_1, F_2, M_1, M_2$ satisfy
	\[
	F_1 = \frac{\eta}{1+\eta} F, \quad
	F_2 = \frac{1}{1+\eta} F, \quad
	M_1 = \frac{\eta}{1+\eta} M, \quad
	M_2 = \frac{1}{1+\eta} M.
	\]
	
	\begin{proof}[Proof of Proposition \ref{prop:infty}]
		To prove this result, we first apply the Arzela-Ascoli theorem for the sequence of smooth solution $\{\mathbf{u}^\varepsilon\}_\varepsilon$ on a close interval $[0, T]$ with any $T > 0$. Then, we extend the convergence at infinity. 
		
		$\bullet$ \textbf{Uniform convergence on $[0, T]$: } First, we check the uniform boundedness of this sequence. For $i= 1, 2$, from Lemma \ref{lem:bound}, one has $E_i^\varepsilon(t) \leq K_i$ for all $t > 0$ and $\varepsilon > 0$. Again by this Lemma, for any $t > 0$, one has 
		\[
		F_i^\varepsilon(t) \leq \max \left\{  F_1^{\varepsilon, 0} + F_2^{\varepsilon, 0}, \ \frac{r \nu_E (K_1 + K_2)}{\mu_F} \right\} \leq C_F^0
		\]
		where $C^0_{F}$ does not depend on $\varepsilon$ since the initial data converge as $\varepsilon$ goes to zero. Similarly, we can apply Lemma \ref{lem:bound} to show that there are positive constants $C^0_{M}, \ C^0_{M^s}$ not depending on $\varepsilon$ such that for any $t > 0$, one has $M_i^\varepsilon(t) < C^0_{M}, \ M^{s, \varepsilon}_i(t) < C^0_{M^s}$. 
		
		Next, we prove that the sequence of derivative $\{\dot{\mathbf{u}^\varepsilon} \}_\varepsilon$ is also uniformly bounded. For any $t > 0$, 
		\[
		\dot{F^{\varepsilon}_1} < r\nu_E K_1 - \frac{1}{\varepsilon} \left( F^\varepsilon_1 - \eta F^\varepsilon_2 \right).
		\]
		We show that $\dfrac{F^\varepsilon_1 - \eta F^\varepsilon_2}{\varepsilon}$ is uniformly bounded on $[0, T]$. Indeed, we have for all $t > 0$ and $\varepsilon > 0$, 
		\[
		\dot{F^\varepsilon_1} - \eta \dot{F^\varepsilon_2} = r\nu_E \left(E^\varepsilon_1 \frac{M^\varepsilon_1}{M^\varepsilon_1 + \gamma M^{s, \varepsilon}_1} - \eta E^\varepsilon_2 \frac{M^\varepsilon_2}{M^\varepsilon_2 + \gamma M^{s, \varepsilon}_2}\right) - \left(\mu_F + \frac{\eta + 1 }{\varepsilon} \right)(F^\varepsilon_1 - \eta F^\varepsilon_2) 
		\]
		\[
		= A^\varepsilon - \left(\frac{\eta + 1 }{\varepsilon} \right)(F^\varepsilon_1 - \eta F^\varepsilon_2),
		\]
		where $A^\varepsilon := r\nu_E \left(E^\varepsilon_1 \dfrac{M^\varepsilon_1}{M^\varepsilon_1 + \gamma M^{s, \varepsilon}_1} - \eta E^\varepsilon_2 \dfrac{M^\varepsilon_2}{M^\varepsilon_2 + \gamma M^{s, \varepsilon}_2}\right) - \mu_F(F^\varepsilon_1 - \eta F^\varepsilon_2) $ is uniformly bounded since we already proved that $\mathbf{u}^\varepsilon$ is uniformly bounded. Then, for any $\varepsilon > 0$, we have $\left|A ^\varepsilon(t)\right | < C$ for any $t > 0$ and some constant $C > 0$. By the Duhamel formula, we obtain 
		\[
		(F^\varepsilon_1 - \eta F^\varepsilon_2)(t) =(F^{\varepsilon, 0}_1 - \eta F^{\varepsilon, 0}_2) e^{-\frac{\eta+1}{\varepsilon} t} + \displaystyle \int_0^t A^\varepsilon(s)e^{-\frac{\eta + 1}{\varepsilon} (t-s)} ds.
		\]
		So for all $t \geq 0$, one has 
		\[
		\frac{|F^\varepsilon_1 - \eta F^\varepsilon_2|(t)}{\varepsilon} \leq \frac{\left|F^{\varepsilon, 0}_1 - \eta F^{\varepsilon, 0}_2\right|}{\varepsilon} e^{-\frac{\eta+1}{\varepsilon} t} + \frac{C}{\eta +1} \left(1-e^{-\frac{\eta + 1}{\varepsilon}t}\right). 
		\]
		For any $t \in [0, +\infty)$ and $\varepsilon > 0$, one has $0 < e^{-\frac{\eta+1}{\varepsilon} t} < 1$. And due to the Assumption \eqref{eqn:initial} for the initial data, the right-hand side is uniformly bounded with respect to $\varepsilon$. Hence, we deduce that $\dot{F^{\varepsilon}_1}$ is uniformly bounded on $[0, T]$. We obtain analogously the uniform boundedness of $\dot{F^\varepsilon_i}, \ \dot{M^\varepsilon_i},$ and $\dot{M^{s, \varepsilon}_i}$. Due to the positivity of system \eqref{eqn:main2patch}, one has $\dot{E^\varepsilon_i}(t) < b C^0_{F_i}$ for all $t > 0$ and $\varepsilon > 0$. 
		
		Since the sequence of derivatives $\{\dot{\mathbf{u}^\varepsilon} \}_\varepsilon$ is uniformly bounded on $[0, T]$, we deduce the equicontinuity of the sequence $\{\mathbf{u}^\varepsilon \}_\varepsilon$. Hence, by the Arzela-Ascoli theorem, this sequence has a uniformly convergent subsequence. We denote its limit $\mathbf{u} = (E_1, F_1, M_1, M^s_1, E_2, F_2, M_2, M^s_2)$. If we multiply system \eqref{eqn:main2patch} with $\varepsilon$ and let it go to zero, we obtain the equalities \eqref{eqn:limit} and system \eqref{eqn:infinity}.
		
		With the initial data satisfying the assumptions in Proposition \ref{prop:infty}, the solution of system \eqref{eqn:infinity} on $(0, +\infty)$ is unique. Since all the subsequence of $\{\mathbf{u}^\varepsilon \}_\varepsilon$ converge to the same limit, we deduce that the whole sequence converges uniformly to this limit on $[0, T]$. 
		
		$\bullet$ \textbf{Extension to $+\infty$: } For all $t \geq 0$, we prove that for all $\delta > 0$, there exists $\varepsilon_0 > 0$ such that for all $\varepsilon \in (0, \varepsilon_0)$, we have $\lVert \mathbf{u}^\varepsilon(t) - \mathbf{u}(t) \rVert < \delta$. 
		
		Indeed, the solution of both \eqref{eqn:main2patch} and \eqref{eqn:infinity} converges to a constant as time $t$ goes to infinity, then there exists a time $T > 0$ large enough and $\varepsilon_1 > 0$ such that for any $\varepsilon \in (0, \varepsilon_1)$ and all $t > T$, one has 
		\[
		\lVert \mathbf{u}^\varepsilon(t) - \mathbf{u}^\varepsilon(T) \rVert < \frac{\delta}{3}, \quad \lVert \mathbf{u}(t) - \mathbf{u}(T) \rVert < \frac{\delta}{3}. 
		\]
		Moreover, we have that the sequence $\{\mathbf{u}^\varepsilon \}_\varepsilon$ converges uniformly to $\mathbf{u}$ in the closed interval $[0, T]$. Thus, there exists a positive value $\varepsilon_0 < \varepsilon_1$ such that for all $\varepsilon \in (0, \varepsilon_0)$,
		$
		\displaystyle \sup_{[0, T]} \lVert \mathbf{u}^\varepsilon - \mathbf{u} \rVert < \frac{\delta}{3}. 
		$
		Hence, we have $\lVert \mathbf{u}^\varepsilon(T) - \mathbf{u}(T) \rVert < \dfrac{\delta}{3}$ and we deduce that  
		\[
		\lVert \mathbf{u}^\varepsilon(t) - \mathbf{u}(t) \rVert \leq \lVert \mathbf{u}^\varepsilon(t) - \mathbf{u}^\varepsilon(T) \rVert +  \lVert \mathbf{u}^\varepsilon(T) - \mathbf{u}(T) \rVert +  \lVert \mathbf{u}(t) - \mathbf{u}(T) \rVert < \delta. 
		\]
		It is clear that for $t \leq T$, one has $\lVert \mathbf{u}^\varepsilon(t) - \mathbf{u}(t) \rVert < \dfrac{\delta}{3} < \delta$. So we obtain the convergence on $[0, +\infty)$. 
	\end{proof}
	
	In the next result, we study the limit system \eqref{eqn:infinity}. 
	\begin{theorem}
		\label{prop:infinity}
		Consider system \eqref{eqn:infinity} with the release function given by
		\begin{itemize}
			\item [(i)] constant release $\Lambda_\infty(t) \equiv \Lambda_\infty$. 
			
			Then there exists $\overline{\Lambda}_\infty > 0$ such that for any $\Lambda_\infty > \overline{\Lambda}_\infty $, system \eqref{eqn:infinity} has a unique equilibrium $ u^0_\infty = \left(0,0,0,0, \frac{\Lambda_\infty}{\mu_s} \right)$ and it is globally asymptotically stable. 
			\item [(ii)] impulsive periodic release $\Lambda_\infty(t) = \displaystyle \sum_{k=0}^{+\infty} \tau \Lambda^\mathrm{per}_\infty \delta_{k\tau}$ with period $\tau$. 
			
			Then there exists $\overline{\Lambda}^\mathrm{per}_\infty > 0$ such that for any $\Lambda^\mathrm{per}_\infty > \overline{\Lambda}^\mathrm{per}_\infty $, all trajectories of \eqref{eqn:infinity} resulting from any non-negative initial data satisfy that $(E_1, E_2, F, M)$ converges to the equilibrium $\mathbf{0}_4 \in \mathbb{R}^4$. 
		\end{itemize}
		
	\end{theorem}
	\begin{proof}
		Firstly,  we show that for $\Lambda_\infty(t)$ large enough such that $\dfrac{M}{M + \gamma M^s} \leq \dfrac{1}{\mathcal{N}}$, then all trajectories of \eqref{eqn:infinity} resulting from any non-negative initial data satisfy that $(E_1, E_2, F, M)$ converges to the equilibrium $\mathbf{0}_4 \in \mathbb{R}^4$. Indeed, consider the first four equations of system \eqref{eqn:infinity} with $\dfrac{M}{M + \gamma M^s}$ replaced by $\dfrac{1}{\mathcal{N}}$, and we denote the equilibrium $(E_1^*, E_2^*, F^*, M^*)$ of this system satisfy 
		\[
		F^* = \frac{r \nu_E (E_1^* + E_2^*)}{\mathcal{N}\mu_F}, \qquad M^* = \frac{(1-r)\nu_E (E_1^* + E_2^*)}{\mu_M},
		\]
		and 
		\[
		(E_1^* + E_2^*) = \frac{E_1^* + E_2^*}{\frac{\eta}{\eta + 1} \left(1-\frac{E_1^*}{K_1} \right) + \frac{1}{\eta + 1} \left(1-\frac{E_2^*}{K_2} \right)}. 
		\]
		This is equivalent to either $E_1^* + E_2^* = 0$ or 
		\[\dfrac{\eta}{\eta + 1} \left(1-\dfrac{E_1^*}{K_1} \right) + \dfrac{1}{\eta + 1} \left(1-\dfrac{E_2^*}{K_2} \right) = 1. \]
		The left-hand side of this equality is smaller than 1 since $E_i^* < K_i$ with $i = 1, 2$, thus we deduce that $E_1^* = E_2^* = F^* = M^* = 0$. Hence, this system has exactly one equilibrium $\mathbf{0}_4$ and all trajectories converge to this steady state by using Theorem 3.1 in Chapter 2 of \cite{SMI}. Then, by applying the comparison Lemma \ref{lem:chaplygin}, we deduce the convergence of system \eqref{eqn:infinity}. 
		
		Analogously to system \eqref{eqn:main2patch}, we have the boundedness for the solution of \eqref{eqn:infinity} and the monotonicity of the system with respect to $\Lambda_\infty$. Therefore, we can deduce the existence of the critical values for both the constant and periodic cases.       
	\end{proof}
	Next, we make a comparison between the previous case and the case where there is no separation between the two sub-populations.
	\subsubsection{The non-separation case}
	\label{subsec:nonseparation}
	When there is no separation between the two sub-populations of mosquitoes, we consider one population $(E, F, M, M^s)$ in a habitat with aquatic carrying capacity $K = K_1 + K_2$. Then $(E, F, M, M^s)$ satisfies the following system 
	\begin{subequations}
		\label{eqn:homo}
		\begin{align}
			\dot E &= bF\left( 1 - \frac{E}{K}\right) - (\nu_E + \mu_E) E, \label{eqn:E} \\
			\dot F &= r\nu_E E \frac{M}{M + \gamma M^s} - \mu_F F, \label{eqn:F} \\
			\dot M & = (1-r)\nu_E E - \mu_M M, \label{eqn:M} \\
			\dot{M}^{s} &= \Lambda - \mu_s M^s . \label{eqn:Ms} 
		\end{align}
	\end{subequations}
	For the constant release, the positive equilibrium $(E^*, F^*, M^*, M^{s*})$ satisfies 
	\[
	M^* = \frac{(1-r) \nu_E}{\mu_M} E^*, \quad M^{s*} = \frac{\Lambda}{\mu_s}, \quad F^* = \frac{r\nu_E}{\mu_F} \frac{E^*}{1 + \frac{\mu_M \gamma \Lambda}{(1-r)\nu_E \mu_s E^*}};
	\]
	and from \eqref{eqn:E}, we deduce that 
	\[
	\frac{br\nu_E}{\mu_F} \frac{E^*}{1 + \frac{\mu_M \gamma \Lambda}{(1-r)\nu_E \mu_s E^*}} \left(1 - \frac{E^*}{K} \right) - (\nu_E + \mu_E) E^* = 0. 
	\]
	This equation has no positive solution if and only if $\Lambda > \overline{\Lambda}_0 = \dfrac{(1-r)\nu_E K \mu_s (1-\mathcal{N})^2}{4\mathcal{N}\mu_M \gamma}$. 
	\begin{remark}
		We can see that in the special case where $K_1 = \eta K_2$, by taking $E = E_1 + E_2$, we can write system \eqref{eqn:infinity}  as system \eqref{eqn:homo} for $(E, F, M, M^s)$ with carrying capacity $K = K_1 + K_2$. Hence, we deduce that $\overline{\Lambda}_\infty = \overline{\Lambda}_0$. This suggests that the critical number of sterile males released in the case with very large diffusion rate is the same as in the non-separation case in \ref{subsec:nonseparation}. 
	\end{remark}
	\subsection{Biological intrinsic values}
	\label{subsec:bio}
	In this section, we compare the critical value of $\Lambda$ corresponding to different values of the parameters namely the birth rate $b$, the death rates $\mu_E, \ \mu_F, \ \mu_M, \ \mu_s$, and the carrying capacities $K_1, \ K_2$.  In this section, we show that the critical value $\overline{\Lambda}$ is monotone with respect to these parameters. To prove this claim, we first define in $\mathbb{R}^7_+$ an order such that 
	$(\mu_E, \mu_F, \mu_M, \mu_s, b, K_1, K_2) \trianglelefteq (\mu_E', \mu_F', \mu_M', \mu_s', b', K_1', K_2')$   if and only if 
	\[
	\mu_E \leq \mu_E', \ \mu_F \leq \mu_F', \ \mu_M \leq \mu_M', \ \mu_s \geq \mu_s', \ b \geq b', \ K_1 \geq K_1', \ K_2 \geq K_2'.  
	\]
	Moreover, we write $(\mu_E, \mu_F, \mu_M, \mu_s, b, K_1, K_2) \vartriangleleft (\mu_E', \mu_F', \mu_M', \mu_s', b', K_1', K_2')$ if the two vectors are not identical. With this order relation, we have the following result
	\begin{theorem}
		\label{thm:bio}
		Consider system \eqref{eqn:main2patch} and the basic offspring number $\mathcal{N} > 1$,  consider the critical values $\overline{\Lambda}$ and $\overline{\Lambda}^\mathrm{per}$ as defined in Theorem \ref{thm:main}, then we have the mappings from $\mathbb{R}^7_+$ to $\mathbb{R}_+$
		\[
		(\mu_E, \mu_F, \mu_M, \mu_s, b, K_1, K_2) \mapsto \overline{\Lambda}, \quad (\mu_E, \mu_F, \mu_M, \mu_s, b, K_1, K_2) \mapsto \overline{\Lambda}^\mathrm{per},
		\]
		are non-increasing with respect to the order $\trianglelefteq$. 
	\end{theorem}
	\begin{proof}
		First, we consider system \eqref{eqn:main2patch} with two sets of parameters 
		\[ \Theta = (\mu_E, \mu_F, \mu_M, \mu_s, b, K_1, K_2), \quad \Theta' = (\mu_E', \mu_F', \mu_M', \mu_s', b', K_1', K_2'),
		\]
		where $\Theta \trianglelefteq \Theta'$. We fix the same value of $\Lambda$ in both cases and consider 
		\[
		\mathbf{u} = (E_1, F_1, M_1, M_1^s, E_2, F_2, M_2, M_2^s), \quad \mathbf{v} = (\widetilde{E_1}, \widetilde{F_1}, \widetilde{M_1}, \widetilde{M_1^s}, \widetilde{E_2}, \widetilde{F_2}, \widetilde{M_2}, \widetilde{M_2^s})
		\]        
		where $\mathbf{u}, \ \mathbf{v}$ are the solutions of \eqref{eqn:main2patch} using the same initial data with the parameters $\Theta, \ \Theta'$, respectively. We have $\dot{\mathbf{u}} = \mathbf{f}_\Theta(\mathbf{u})$, and $\dot{\mathbf{v}} = \mathbf{f}_{\Theta'}(\mathbf{v}) \preceq \mathbf{f}_\Theta(\mathbf{v})$ in the subset $\{0 \leq E_1 \leq K_1\} \cap \{0 \leq E_2 \leq K_2\}$ of $\mathbb{R}^8_+$. Moreover, functions $\mathbf{f}_\Theta$ and $\mathbf{f}_{\Theta'}$ satisfy the assumptions in Lemma \ref{lem:monotonicity}, then by applying this lemma, we obtain that $\mathbf{v} \preceq \mathbf{u}$ for the same initial data, so
		\[
		E_i(t) \geq \widetilde{E_i}(t), \ F_i(t) \geq \widetilde{F_i}(t), \ M_i(t) \geq \widetilde{M_i}(t) \qquad \text{for all } t > 0, \ i = 1, 2.  
		\]
		On the other hand, for any $\Lambda > \overline{\Lambda}_\Theta$, by Theorem \ref{thm:main} we have that $E_i(t), \ F_i(t), \ M_i(t)$ converge to zero as $t$ goes to infinity. As a consequence of the above inequalities, we deduce that $\widetilde{E_i}(t), \ \widetilde{F_i}(t), \ \widetilde{M_i}(t)$ also converge to zero for all initial data. So $\Lambda > \overline{\Lambda}_{\Theta'}$, and we can deduce that $\overline{\Lambda}_\Theta \geq \overline{\Lambda}_{\Theta'}$. 
	\end{proof}
	\section{Numerical simulations \label{sec:numeric}}
	Following \cite{DUF, STR19}, we consider the parameters as in Table \ref{tab:param}. 
	\begin{table}
		\centering
		\caption{Parameter values of \textit{Aedes albopictus} mosquitoes used for the numerical simulation}
		\label{tab:param}
		\begin{tabular}{|c|c|c|c|}
			\hline
			\textbf{Symbol} & \textbf{Description} & \textbf{Value} & \textbf{Unit} \\
			\hline
			$b$ & Birth rate of fertile females & 10 & $\text{day}^{-1}$ \\
			\hline
			$\nu_E$ & Emerging rate of viable eggs & 0.08 &  $\text{day}^{-1}$ \\
			\hline
			$\mu_E$ & Death rate of aquatic phase & 0.05 & $\text{day}^{-1}$ \\
			\hline 
			$\mu_F$ & Female death rate & 0.1 & $\text{day}^{-1}$\\
			\hline 
			$\mu_M$ & Wild male death rate & 0.14 & $\text{day}^{-1}$ \\
			\hline 
			$\mu_s$ & Sterile male death rate & 0.14 & $\text{day}^{-1}$ \\
			\hline
			$K_1$ & Carrying capacity of aquatic phase in patch 1 & 200 & \_ \\ 
			\hline
			$K_2$ & Carrying capacity of aquatic phase in patch 2 & 180 & \_ \\ 
			\hline 
			$\gamma$ & Mating competitiveness of sterile male & 1 & \_ \\
			\hline
			$r$ & Ratio of female hatch & 0.5 & \_ \\
			\hline
			$\alpha$ & Ratio between diffusion rates of sterile males and female & 0.5 & \_ \\
			\hline
			$\beta$ & Ratio between diffusion rates of sterile males and female & 0.8 & \_ \\
			\hline
		\end{tabular}
	\end{table} 
	\subsection{Trajectories and Equilibria}
	We fix the moving rate $d_{12} = 0.06, \ d_{21} = 0.04$ ($\text{day}^{-1}$), and plot the numerical solutions of system \eqref{eqn:main2patch} with different releases functions $\Lambda(t)$. 
    In the case $\Lambda = 0$, system \eqref{eqn:natural} has a unique equilibrium 
    \[
    (E_1^+, F_1^+, M_1^+, E_2^+, F_2^+, M_2^+) = (192.63, 67.92, 50.29, 174.83, 79.06, 54.69).
    \]
    To highlight the global stability of equilibria, we numerically solve the system with initial data in three levels: close to zero, intermediate, and close to the positive equilibrium when $\Lambda = 0$. For the level close to zero, the initial data is taken between $1\%$ to $10\%$ of the value of the above equilibrium, between $40\%$ and $50\%$ for the intermediate level, and between $90\%$ to $100\%$ for the third case. More precisely, we take \[(E_1^0, F_1^0, M_1^0, M_1^{s,0}, E_2^0, F_2^0, M_2^0, M_2^{s,0}) \in \{(2,5,4,0,3,5,3,0),\ (80,30,20,0,70,30,30,0), \ (160,60,50,0,155,70,50,0)\}.\]
    When the largest wild mosquito population was less than $10^{-2}$, we considered the wild population to be eliminated.
	The following section presents several numerical simulations showing the trajectories and approximated equilibria according to different release strategies.
	
	\subsubsection{Constant continuous releases}
	We take three different constant values of $\Lambda \in \{0, 200, 500\}$ ($\text{day}^{-1}$). The initial density of sterile males is equal to zero. We approximate the positive equilibria in each case and plot the trajectories of $E_1$ and $E_2$ in Figures \ref{fig:trajectory} according to different values of $\Lambda$. We observe the following: 
	\begin{itemize}[leftmargin = 0.5 cm]
		\item When $\Lambda = 0$, there is one positive equilibrium 
		\[
		(E_1^*, E_2^*) = (192.62, 174.82).
		\]
		All positive trajectories converge to the positive steady state $(E_1^*, E_2^*)$. 
		\item  When $\Lambda = 200$ ($\text{day}^{-1}$), there are two positive equilibria
		\[
		(e_1^*, e_2^*) =  (17.29, 49.98), \qquad (E_1^*, E_2^*) = (85.79, 130.02).
		\]
		All positive trajectories also converge to the larger positive steady state $(E_1^*, E_2^*)$. 
		\item  When $\Lambda = 500$ ($\text{day}^{-1}$), there is no positive equilibrium. All the trajectories converge to the zero equilibrium.
	\end{itemize}
	This validates the result in Theorem \ref{thm:main} that when $\Lambda$ exceeds some critical value, zero is the unique equilibrium of system \eqref{eqn:main2patch}.
	The observation for $\Lambda = 0$ illustrates the result in Theorem \ref{thm:equi1} that there is one positive equilibrium and it is globally asymptotically stable. The introduction of sterile males ($\Lambda = 200 > 0$) reduces the value of the positive steady state (see Figure \ref{fig:trajectory1}), and when $\Lambda = 500$ ($\text{day}^{-1}$) exceeds some critical value (at most equal to 500), all trajectories converge to the zero equilibrium (see Figure \ref{fig:trajectory2}). This illustrates the first point of Theorem \ref{thm:main}. To approximate the critical value of $\Lambda$, we provide some numerical bifurcation diagrams in Section \ref{subsec:bif}. 
	
	\begin{figure}
		\centering
		\begin{subfigure}{\textwidth}
			\centering
			\includegraphics[width = \textwidth]{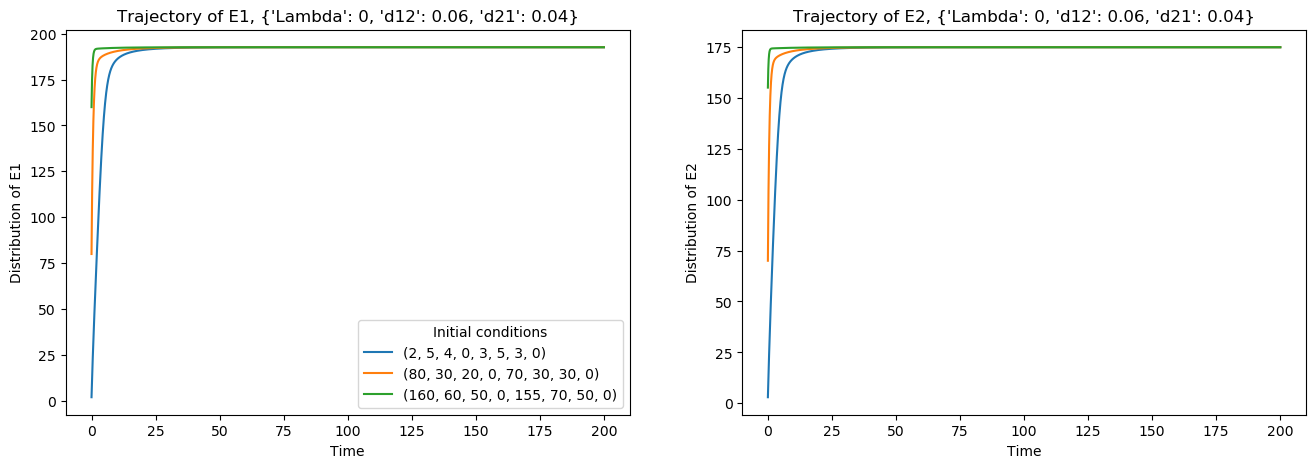}
			\caption{$\Lambda = 0$}
			\label{fig:trajectory0}
		\end{subfigure}\\ 
		\begin{subfigure}{\textwidth}
			\centering
			\includegraphics[width = \textwidth]{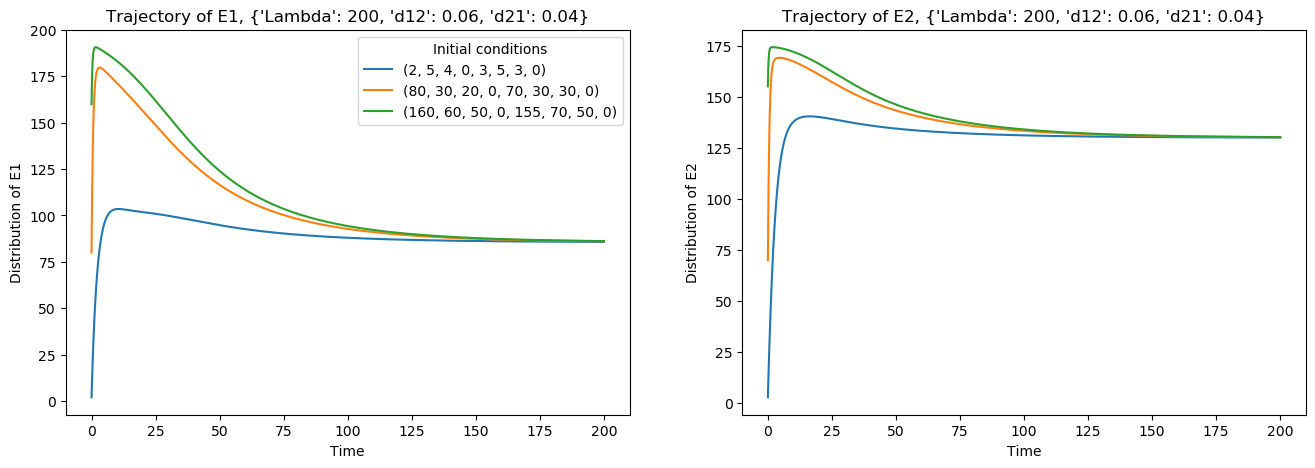} 
			\caption{$\Lambda = 200$ ($\text{day}^{-1}$).}
			\label{fig:trajectory1}
		\end{subfigure} \\      
		\begin{subfigure}{\textwidth}
			\centering
			\includegraphics[width = \textwidth]{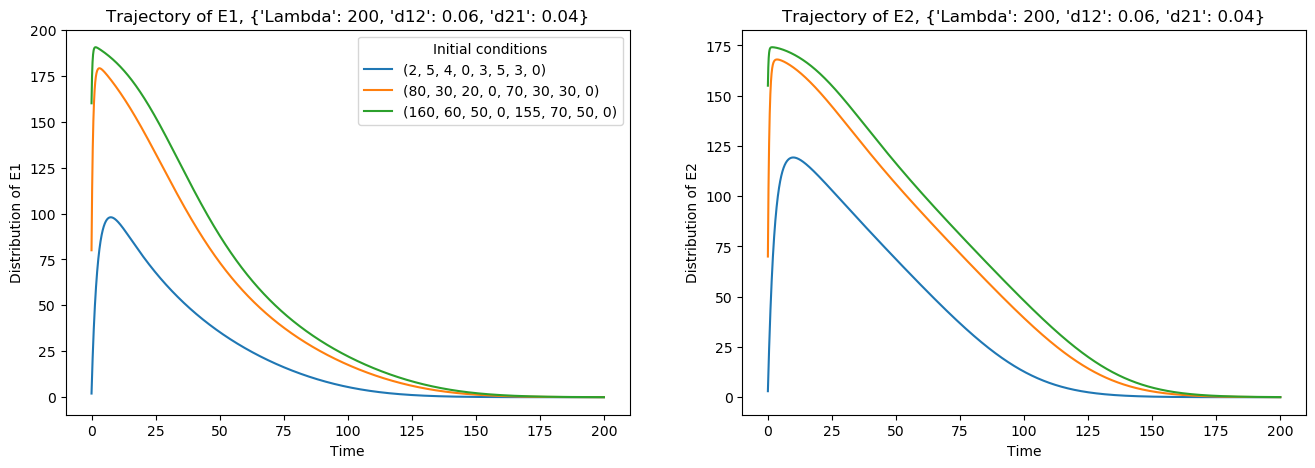} 
			\caption{$\Lambda = 500$ ($\text{day}^{-1}$).}
			\label{fig:trajectory2}
		\end{subfigure}
		\caption{Trajectories of $E_1$ and $E_2$ in the constant release case with diffusion rates $d_{12} = 0.06, \ d_{21} = 0.04$ ($\text{day}^{-1}$).}
		\label{fig:trajectory}
	\end{figure}
	\subsubsection{Periodic impulsive releases}
	In this part, we consider the periodic impulsive releases with $\Lambda(t)$ defined in \eqref{eqn:Lambdaper}, with $\Lambda^\mathrm{per}$ equal to $200$ and $300$ ($\text{day}^{-1}$), the period $\tau = 10$ (days). The trajectories of $E_1, \ E_2$ shown in Figure \ref{fig:trajectory_per} converge to the periodic solution when $\Lambda^\mathrm{per} = 200$ ($\text{day}^{-1}$) and go to zero when $\Lambda^\mathrm{per} = 300$ ($\text{day}^{-1}$). This illustrates the second point of Theorem \ref{thm:main} that when the number of sterile males released exceeds a critical value $\overline{\Lambda}^\mathrm{per}$, the wild populations of mosquitoes in both areas reach elimination.     
	\begin{figure}
		\centering
		\begin{subfigure}{\textwidth}
			\centering
			\includegraphics[width = \textwidth]{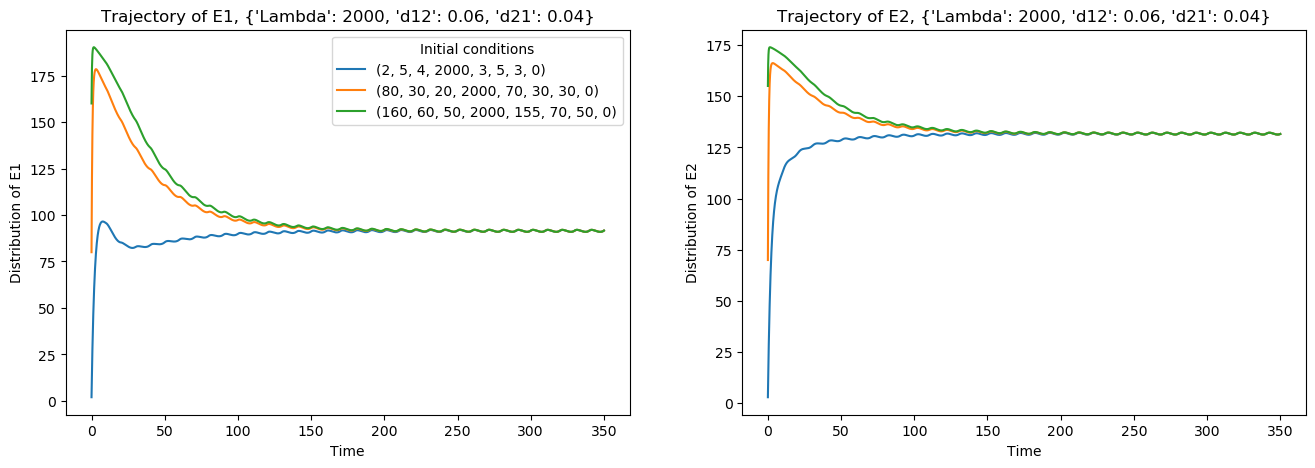}
			\caption{$\Lambda^\mathrm{per} = 200$ ($\text{day}^{-1}$).}
			\label{fig:trajectory_per1}
		\end{subfigure}\\ 
		\begin{subfigure}{\textwidth}
			\centering
			\includegraphics[width = \textwidth]{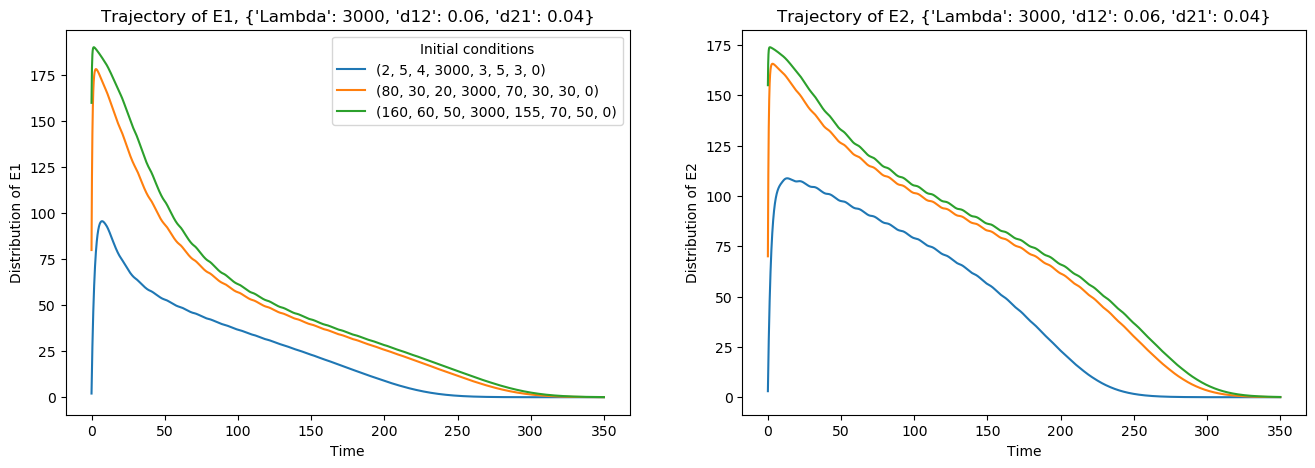}
			\caption{$\Lambda^\mathrm{per} = 300$ ($\text{day}^{-1}$).}
			\label{fig:trajectory_per2}
		\end{subfigure}
		\caption{Trajectories of $E_1$ and $E_2$ in the periodic release case with period $\tau = 10$ (days), diffusion rates $d_{12} = 0.06, \ d_{21} = 0.04$ ($\text{day}^{-1}$).}
		\label{fig:trajectory_per}
	\end{figure}
	\subsection{Critical values and bifurcation}
	\label{subsec:bif}
	Our aim in this section is to approximate the critical value of $\Lambda$ where the bifurcation occurs.  
	\subsubsection{Bifurcation diagram in the constant release case}
        \begin{figure}
		\centering
		\begin{subfigure}{0.45 \textwidth}
			\centering
			\includegraphics[width = \textwidth]{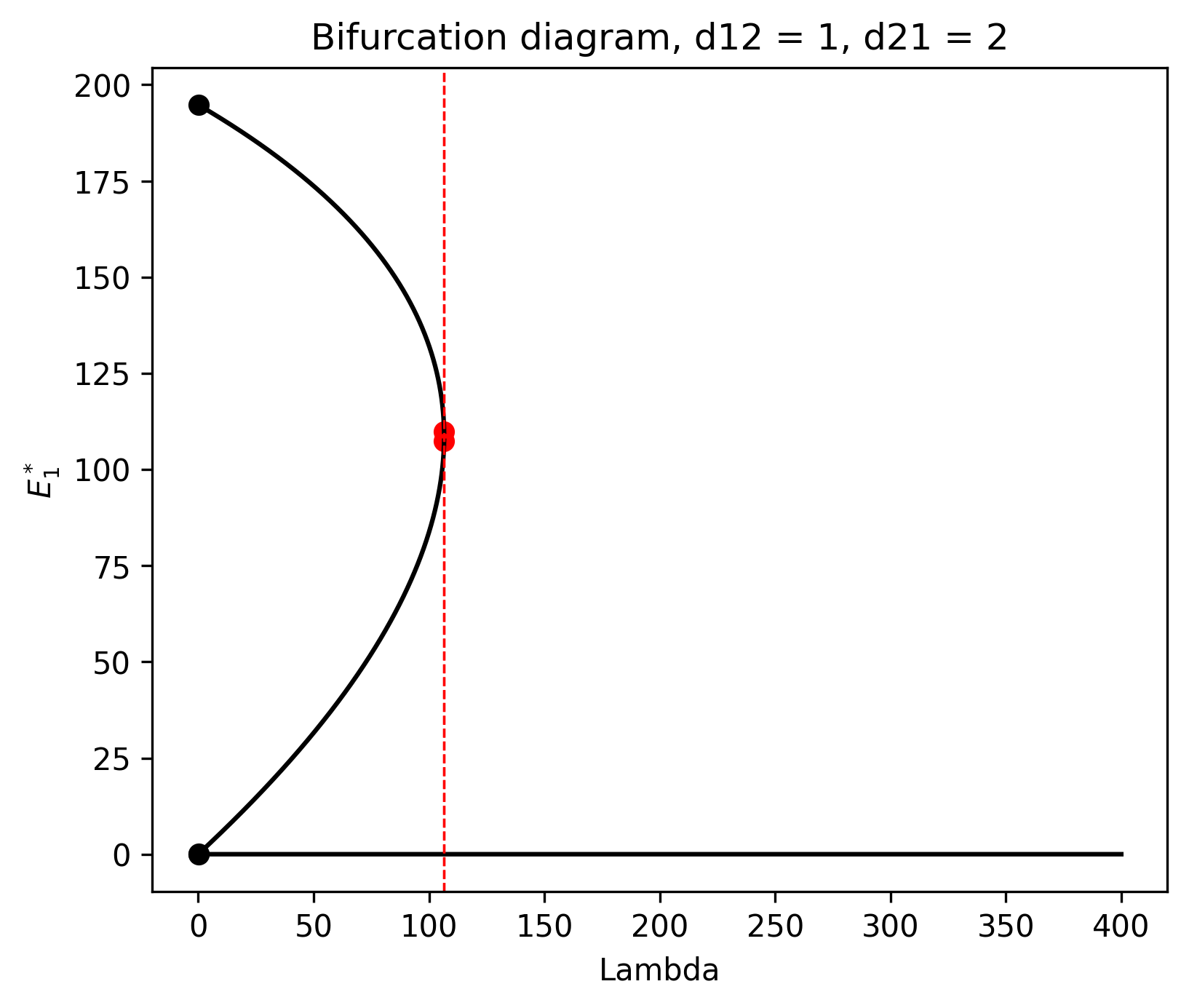}
		\end{subfigure}
		\begin{subfigure}{0.45 \textwidth}
			\centering
			\includegraphics[width = \textwidth]{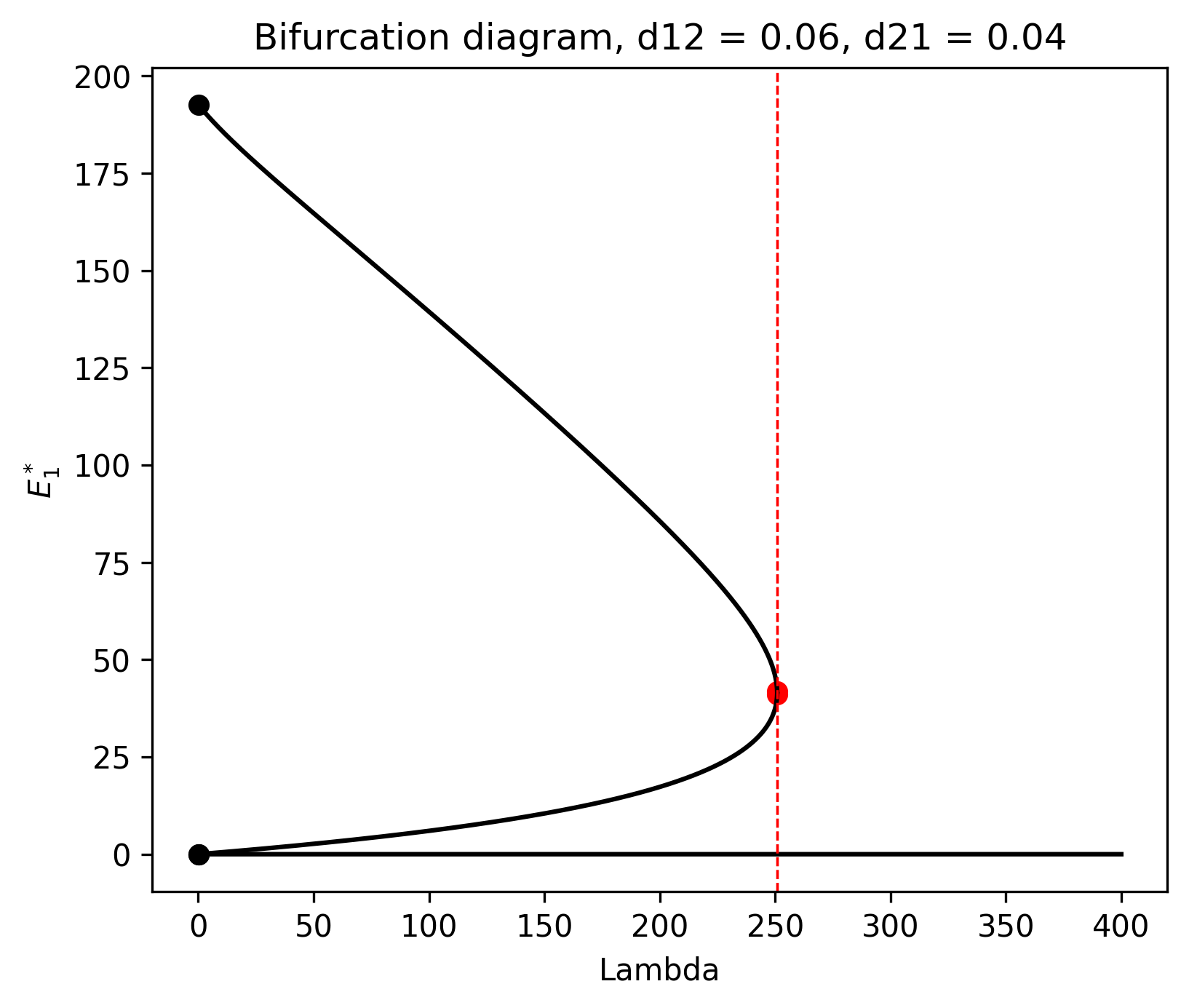}
		\end{subfigure}
		\caption{Bifucation diagrams of $E_1^*$ with parameter $\Lambda$ in the constant continuous release case. }
		\label{fig:bifurcation}
	\end{figure}
	We solve a system of nonlinear stationary problems $\mathcal{F}(u; \Lambda)=0$ for all values of the parameter $\Lambda$, knowing that the solutions are continuous with respect to $\Lambda$. Solving by numerical approximations can be done using numerical continuation methods (see \cite{RHE}). 
	
	Here we present the simplest method called \textit{Natural Parameter Continuation} (incremental methods, see \cite{RHE}):  Iteratively find approximate roots of $\mathcal{F}(u;\Lambda) = 0$ for several values of $\Lambda_i$ with index $i \in \mathbb{N}^*$. The root of step $i$ is used as an initial guess for the numerical solver at step $i+1$. The first initial guess is the root for the smallest $\Lambda$. To approximate the critical value $\overline{\Lambda}$ in the constant case and examine what happens when $0 < \Lambda  \leq \overline{\Lambda}$, we draw the bifurcation diagram for $\Lambda \in [0.1, 500]$. The initial positions of the numerical continuation are taken at the approximated equilibria when $\Lambda = 0.1$. 
	
	We obtain the bifurcation diagrams in Figure \ref{fig:bifurcation} for two scenarios. We observed that the critical value of $\Lambda$ decreases when the diffusion rates increase. 
	\begin{itemize}[leftmargin = 0.5 cm]
		\item For $d_{12} = 1, \ d_{21} = 2$, the critical value $\overline{\Lambda} = 106.45$ ($\text{day}^{-1}$).
		\item For $d_{12} = 0.06, \ d_{21} = 0.04$, the critical value $\overline{\Lambda} = 250.88$ ($\text{day}^{-1}$). 
	\end{itemize}
    \begin{figure}
    	\centering
    	\begin{subfigure}{0.48 \textwidth}
    		\centering
    		\includegraphics[width = \textwidth]{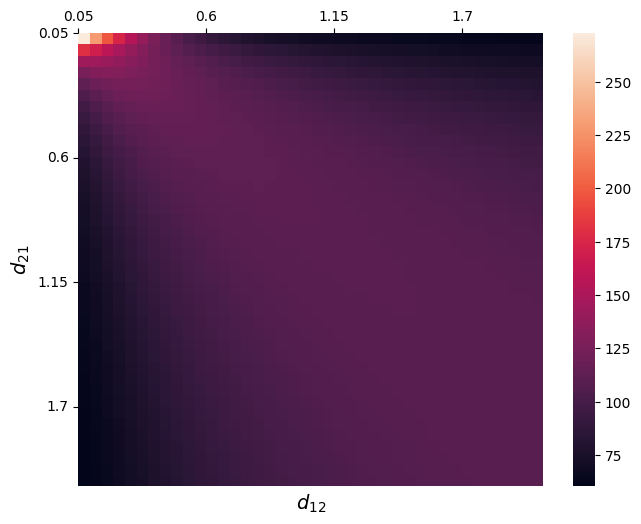}
    		\subcaption{Heatmap depicts values for $\overline{\Lambda}$}
    		\label{fig:heatmap}
    	\end{subfigure}
    	\begin{subfigure}{0.48 \textwidth}
    		\centering
    		\includegraphics[width = \textwidth]{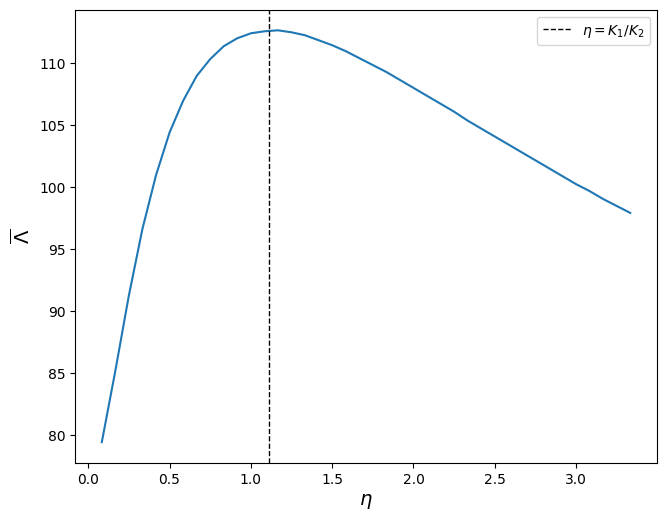}
    		\subcaption{Values of $\overline{\Lambda}$ change with respect to $\eta = d_{21}/d_{12}$}
    		\label{fig:ratio}
    	\end{subfigure}
    	\caption{Values of $\overline{\Lambda}$ varies with respect to diffusion rates $d_{12}, \ d_{21} \in [0.05,2]$. }
    	\label{fig:lambda_d12}
    \end{figure}
	Values of $\overline{\Lambda}$ varying with respect to different values of $d_{12}, \ d_{21} \in [0.05,2]$ were depicted in Figure \ref{fig:heatmap}. We observe that as the diffusion rates get larger, the critical value $\overline{\Lambda}$ is decreasing and converges to some value $\overline{\Lambda}_\infty$. This validates the result provided by Proposition \ref{prop:infinity} where $\overline{\Lambda}_\infty$ is the critical value of $\Lambda$ corresponding to system \eqref{eqn:infinity}. 
		
	By considering the case $d_{21} = \eta d_{12}$, we plot values of $\overline{\Lambda}$ corresponding to different ratios $\eta$ while fixing $d_{12} = 0.6$  in Figure \ref{fig:ratio}. We can see that $\overline{\Lambda}$ reach a local maximum as $\eta$ passes through the ratio $\frac{K_1}{K_2}$. In practice, since the mosquitoes will likely move to areas with more breeding sites, the ratio $\eta = \frac{d_{21}}{d_{12}}$ is close to the carrying capacity $\frac{K_1}{K_2}$. This indicates that intervening on the breeding sites to increase the difference between these two ratios can help decrease the number of sterile males needing release.
	
	Moreover, as $d_{12}$ get larger, this value of $\overline{\Lambda}$ gets closer to the critical value $\overline{\Lambda}_0$ of the system when there is no separation between the two sub-populations defined in \ref{subsec:nonseparation}.
	\subsubsection{Comparison of release strategies}
        \begin{figure}
		\centering
		\includegraphics[width = \textwidth]{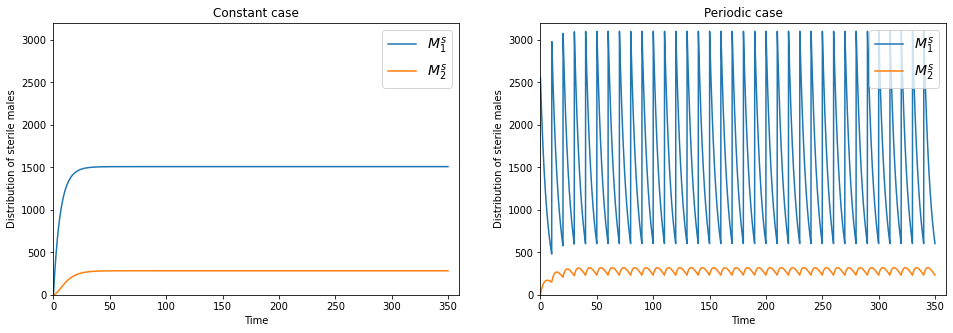}
		\caption{Densities of $M_1^s$ and $M_2^s$ in both cases. \textit{Left:} constant continuous releases with $\Lambda = \overline{\Lambda} = 250.88$ ($\text{day}^{-1}$), \textit{Right:} periodic impulsive releases with $\Lambda^\mathrm{per} = \overline{\Lambda}^\mathrm{per} = 255.15$ ($\text{day}^{-1}$).}
		\label{fig:comparisonMs}
	\end{figure}
 
	In practice, the strategy using impulsive releases is more realistic than the constant strategy. In this section, we make a comparison between these two strategies. 
	
	For the fixed diffusion rates $d_{12} = 0.06, \ d_{21} = 0.04$, we approximated the critical number of sterile males released in both cases using the method in \ref{subsec:bif} 
	\begin{itemize}
		\item When $\Lambda(t) \equiv \Lambda$ constant, the critical value $\overline{\Lambda} \approx 250.88$ ($\text{day}^{-1}$); 
		\item When $\Lambda(t) = \displaystyle \sum_{k=0}^{+\infty} \tau \Lambda^\mathrm{per} \delta_{k\tau}$ with period $\tau = 10$, the critical value of $\Lambda^\mathrm{per}$ is $\overline{\Lambda}^\mathrm{per} \approx 255.15$ ($\text{day}^{-1}$). 
	\end{itemize}
	We can see that $ \overline{\Lambda}$ and $\overline{\Lambda}^\mathrm{per}$ are consistent. We also present numerical simulations in both cases with the same total amount of sterile males released where $\Lambda^\mathrm{per} = \Lambda = 300$. The densities of sterile males in both cases are shown in Figure \ref{fig:comparisonMs}. In the constant release case, one observes that the density of sterile males in both zones converges to an equilibrium $(M_1^{s,*}, M_2^{s,*})$. In the periodic case, Figure \ref{fig:comparisonMs} illustrates the results provided in Lemma \ref{lem:periodic} in which the density of sterile males converges to a periodic solution that is bounded from below.
	
	We obtained in Figure \ref{fig:comparisonEFM} that in both cases, the wild mosquito population reaches elimination at time $t \approx 300$. Again we can see that the two strategies provide the same performance.

    \begin{figure}
	\centering
	\includegraphics[width = \textwidth]{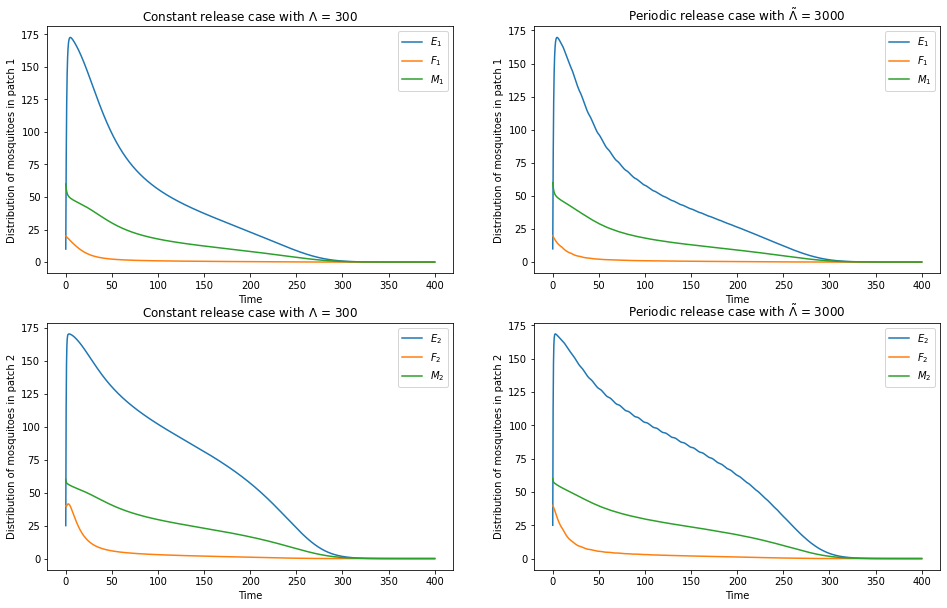}
        \caption{Densities of wild mosquitoes in two patches in both cases. \textit{Left:} constant continuous release with $\Lambda = 300$ mosquitoes released per day, \textit{Right:} periodic impulsive releases with $\Lambda^\mathrm{per} = 300$ mosquitoes released at the beginning of each time period.}
	\label{fig:comparisonEFM}
    \end{figure}
	\section{Discussion and conclusion}
	The existence of some hidden areas (e.g. crab burrows) that can not be accessed by the SIT hinders the population from reaching elimination. Without the implementation of this technique, Theorem \ref{thm:equi1} showed that the wild populations in both areas are persistent and converge towards the unique positive equilibrium (see Figure \ref{fig:trajectory0}) and are independent of the diffusion rates between them. The main results obtained in the present work indicated that with a sufficient number of sterile males released, the SIT succeeds in driving both sub-populations to extinction. We investigated both continuous constant releases and impulsive periodic releases in Theorem \ref{thm:main}. The two strategies provided almost similar performance but the periodic release is more realistic in practice. The idea in our proof can also be used to design a feedback release strategy and this could be studied in future works. 
	
	The results in Theorem \ref{thm:bio} also pointed out that the critical numbers of released sterile males are monotone with respect to the biological parameters of the population (see Section \ref{subsec:bio}). A population with a larger birth rate and a bigger environment carrying capacity requires more sterile males to reach elimination. A larger death rate in any compartment of the wild mosquitoes reduces this critical value, and on the contrary, a larger death rate for the sterile males increases this value. From the control measure point of view, one can lower the threshold value of $\Lambda$ by killing mosquitoes to increase mosquito mortality or removing breeding sites, even only in the accessible zone to reduce the carrying capacity ($K_1$). This indicates that other conventional control measures can be combined with the SIT control to make the SIT elimination threshold easier to attain.
	
	Moreover, the critical number of sterile males also depends on the diffusion rates between the treated area and the inaccessible zone. More precisely, if the diffusion rates are large, this system approaches the case when there is no separation between two sub-populations (see Theorem \ref{prop:infinity}). Numerically, we showed that the larger the values of diffusion rates, the smaller the threshold we need to exceed to obtain elimination (see Figure \ref{fig:lambda_d12}).   This also showed that when the movement is at a low level, the leak of wild mosquitoes from the inaccessible area impedes the eradication in the treated zone and it requires a larger number of sterile males to break through this obstacle. In practice, this could be an unrealistic amount of sterile mosquitoes. It is not surprising that the scenario with larger diffusion between two areas is better since more sterile males can arrive at the unreachable zone. Further study on dispersal rates of mosquitoes is necessary to estimate the release rates for SIT elimination in the presence of inaccessible zones.
	\section*{Acknowledgments}
	This study is supported by the \textbf{STIC AmSud project BIO-CIVIP 23-STIC-02}.  
	\vskip 0.2cm
	
	\begin{minipage}{0.15\textwidth}
		\includegraphics[height = 1.2 cm]{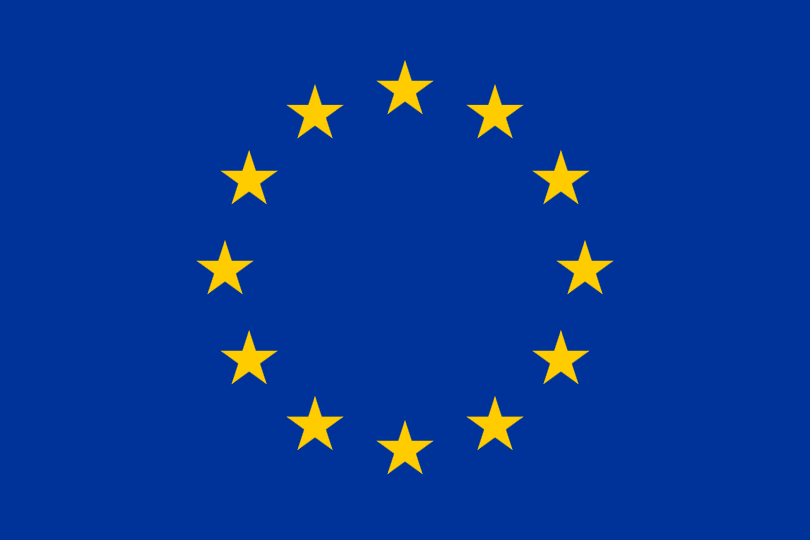} 
	\end{minipage}
	\begin{minipage}{0.8\textwidth}
		{N. Nguyen has received funding from the European Union's Horizon 2020 research and innovation program under the Marie Sklodowska-Curie grant agreement No 945332.}
	\end{minipage}
	\bibliographystyle{ieeetr}
	\bibliography{bib_p}
\end{document}